\documentclass[12pt
 ]{amsart}
\usepackage{url,ulem}
\usepackage[bookmarks=false,unicode,colorlinks,urlcolor=red,citecolor=red,linkcolor=blue]{hyperref}
\usepackage{geometry}
\usepackage{titlesec}
\usepackage{amsmath,amssymb,amsfonts,amsthm,mathrsfs,eucal,dsfont,verbatim}
\usepackage{soul}

\def\dist{\mathop{\rm dist}\nolimits}

\newcommand{\RR}{\mathds{R}}
\newcommand{\R}{ \mathds{R}^{d}}
\newcommand{\Rd}{{\mathds{R}^d}}
\newcommand{\rd}{{\mathds{R}^d}}
\newcommand{\N}{\mathbb{N}}

\newcommand{\I}{\mathds 1}
\newcommand{\scalp}[2]{#1\cdot#2}

\newcommand{\indyk}[1]{{\mathds 1}_{#1}}

\newcommand{\prt}{\partial}

\newcommand{\sphere}{ \mathbb{S}}

\newcommand{\modgener}{{\mathcal A}^\#_t}
 \def\dist{\mathop{\rm
    dist}\nolimits} \def\diam{\mathop{\rm diam}\nolimits}
\newtheorem{lemma}{Lemma}[section]
\newtheorem{theorem}[lemma]{Theorem}
\newtheorem{proposition}[lemma]{Proposition}
\newtheorem{corollary}[lemma]{Corollary}
\newtheorem{definition}{Definition}[section]

\newcounter{conum} 

\usepackage{color}

\newcommand{\vica}{\color{blue}}
\newcommand{\kb}{\color{green}}
\newcommand{\normal}{\color{black}}
\newcommand{\kbm}[1]{\marginpar{\scriptsize \textcolor{green}{#1}}}

\renewcommand{\star}{\circledast}

\newcommand{\real}{\mathds{R}}

\newcommand{\RRR}{\mathrm{Re}\,}

\numberwithin{equation}{section}
\titleformat{\section}
       {\bfseries}{\thesection}{1em}{}

\titleformat{\subsection}
       {\itshape}{\thesubsection}{1em}{}
\begin{document}

\title{Heat kernel of anisotropic nonlocal operators}
\author[K. Bogdan]{Krzysztof Bogdan}
\address{ Faculty of Pure and Applied Mathematics,
Wroc\l aw University of Science and Technology,
Wyb. Wyspia\'nskiego 27, 50-370 Wroc\l aw, Poland}
\email{bogdan@pwr.edu.pl}
\author[P. Sztonyk]{Pawe{\l} Sztonyk}
\address{ Faculty of Pure and Applied Mathematics,
Wroc\l aw University of Science and Technology,
Wyb. Wyspia\'nskiego 27, 50-370 Wroc\l aw, Poland}
\email{Pawel.Sztonyk@pwr.edu.pl}
\author[V. Knopova]{Victoria Knopova  }
\address{V. M. Glushkov Institute of Cybernetics, 40, Acad. Gluskov Ave., 03187 Kiev, Ukraine /University of Wroclaw,  plac Uniwersytecki 1, 50-137 Wroc\l aw, Poland }
\email{vicknopova@gmail.com}
\thanks{The research of K. Bogdan and V. Knopova was partially supported by the NCN grant 2014/14/M/ST1/00600.}
\date{\today}
\maketitle

\begin{abstract}
We construct and estimate the fundamental solution of  highly ani\-so\-tro\-pic space-inhomogeneous integro-differential operators. We use the Levi method. We give applications to the Cauchy problem for such operators.
\end{abstract}

\footnotetext{2010 {\it MS Classification}:
Primary 47D03;
Secondary 60J35.
}
\section{Introduction and  main results}\label{sec:iar}

Semigroups of operators are at the core of mathematical analysis. They describe evolutionary phenomena, resolve parabolic differential equations and have many connections to   spectral theory and integro-differential calculus.
In the paper we focus on Markovian semigroups--that is, probability kernels  satisfying the Chapman-Kolmogorov equation. They  model dissipation of mass and underwrite stochastic calculus of Markov processes.
We will construct  semigroups
determined by integral kernel $\nu(x,A)$, called the L\'evy kernel, and
interpreted as the intensity of the occurrence of dislocations of mass, or jumps, from position $x\in \Rd$ to the set $x+A\subset \Rd$.

The construction of the semigroup from the L\'evy kernel
is intrinsically difficult when
$\nu$ is rough, just like the construction of a flow from a non-Lipschitz direction field or a diffusion from a second order elliptic operator
with merely bounded or degenerate coefficients.
In this paper under appropriate assumptions on $\nu$ we obtain the semigroup  and estimate its integral kernel $p_t(x,y)$, called the heat kernel or the fundamental solution or the transition probability density,
and we prove regularity and uniqueness of the kernel.
Our results are analogues of the construction and estimates of the heat kernel for the second order elliptic operators with rough or degenerate coefficients.

A unique feature of our contribution is that we deal with
 highly anisotropic L\'evy kernels, meaning that $\nu(x,A)$ may
 vanish in certain jump directions.  In fact $\nu$
 may be
concentrated on a set of directions of Lebesgue measure zero.
We should note that despite recent rapid accumulation of estimates of heat kernels of nonlocal integro-differential L\'evy-type generators with kernels $\nu(x,A)$ so far there were virtually none on generators with highly anisotropic kernels.
 A notable exception  are the papers
 by Sztonyk et al. \cite{MR2320691, MR3357585, 2014arXiv1403.0912K, Sz2016}
 but they only concern translation invariant generators and convolution semigroups, for which the existence and many properties  follow by  Fourier methods.  We also mention the estimates of  anisotropic non-convolution heat kernels $p_t(x,y)$ 
given in \cite{MR2718253} and \cite{MR3089797},  however these are obtained under the assumption that the heat kernel   exists, without constructing it.

Specifically, we
consider jump kernels $\nu(x, dz)$ comparable to the L\'evy measure $\nu_0(dz)$
of a symmetric anisotropic $\alpha$-stable L\'evy process in $\Rd$. Here and below we always  assume that $0<\alpha<2$ and $d=1,2,\ldots$. For important technical reasons we also require 
H\"older continuity in $x$   of the Radon-Nikodym derivative
$\nu(x,dz)/
\nu_0(dz)$.
Recall that the L\'evy measure $\nu_0$ of the $\alpha$-stable L\'evy process
has the form of a product measure
in  polar coordinates: $\nu_0(drd\theta)=r^{-1-\alpha}dr\mu_0(d\theta)$. The anisotropy referred to above means that the spherical  marginal $\mu_0$ may even be singular with respect to the surface measure on the unit sphere. In fact
we assume that $\nu_0$ (and $\nu$)
have
Haussdorff-type regularity near  the unit sphere.
  The order $\gamma$ of the regularity is a fundamental factor in our development: we require $\alpha+\gamma>d$.

To construct the heat kernel $p$
from the L\'evy kernel
$\nu$ we use the parametrix method. It is a general approach, which starts from an implicit equation and {\emph a} first approximation $p^0$  for $p$. Iterating the equation  produces an explicit (parametrix) series.
The series formally solves the equation  but the actual proof require delicate analysis of convergence, which critically depends on the
apposite
choice of the first approximation $p^0$. The method  was  proposed by E.~Levi \cite{zbMATH02644101} to solve an elliptic Cauchy problem.  It was then extended by Dressel \cite{MR0003340} to parabolic systems  and  by Feller \cite{zbMATH03022319} to parabolic operators  perturbed by bounded non-local operator.
Further  developments were given in papers of Drin' \cite{MR0492880}, Eidelman and  Drin' \cite{MR616459}, Kochube\u\i{} \cite{MR972089} and Kolokoltsov \cite{MR1744782}. We also refer the reader to the monograph by Eidelman, Ivasyshin and Kochube\u\i{} \cite{MR2093219} and to the classical monograph of Friedman \cite{MR0181836} on the second-order parabolic differential operators.
The parametrix method  has a  variant called the perturbation or Duhamel  formula and series. The variant is appropriate for adding a ``lower order'' term to the generator of a given  semigroup and the role of the first approximation is played by the ``unperturbed'' semigroup. This is, however, not
the situation in the present paper, because $\nu(x,dz)-\nu_0(dz)$ is not of ``lower order''  in comparison with $\nu_0(dz)$.

For  recent developments in  the parametrix and perturbation methods  for  nonlocal operators  we refer the interested reader  to
Bogdan and Jakubowski \cite{MR2283957}, Knopova and Kulik \cite{2014arXiv1412.8732K}, \cite{2013arXiv1308.0310K}, Ganychenko, Knopova and Kulik \cite{MR3456146}, Kulik \cite{1511.00106}, Chen and Zhang \cite{MR3500272}, Kim, Song and Vondracek \cite{2016arXiv160602005K} and K{\"u}hn \cite{2017arXiv170200778K}.
A different Hilbert-space approach was developed in    Jacob \cite{MR1254818,MR1917230}, Hoh \cite{MR1659620} and
B\"ottcher \cite{MR2163294, MR2456894} and relies on the symbolic calculus, see also Tsutsimi \cite{MR0380172,MR0499861} and   Kumano-go \cite{MR1414739}.
We should note again that
the listed papers assume that $\nu(x,dz)/dz$
is locally comparable with a radial function.   This is what we call the isotropic setting. The anisotropic setting has different methods and very few results. Here we show how to handle space-dependent anisotropic generators using suitable majorization and recent precise estimates for stable convolution semigroups.

After verifying that
the parametrix series
representing $p_t(x,y)$ is
convergent, one is challenged to prove
that $p$ is indeed the fundamental solution, in particular that it is Markovian and the {\it generator} of the semigroup coincides with the integro-diferential {\it operator} defined by $\nu$ for sufficiently large class of functions. This is a complicated task.
 The method described by Friedman \cite{MR0181836}  consists in (1) proving that 
$p_t(x,y)$ gives solutions to the respective Cauchy problem for the operator and (2) using the maksimum principle for the operator.  This approach is extended to rather isotropic nonlocal operators by Kochube\u\i{}  \cite{MR972089} and further developed in the isotropic setting by Chen and Zhang \cite{MR3500272} and by Kim, Song and Vondracek \cite{2016arXiv160602005K}.
 In our work we indeed profited a lot by following the outline of Kochube\u\i{}     \cite{MR972089}.
Another  method, based on suitable approximations of the fundamental solutions was developed by Knopova and  Kulik \cite{2013arXiv1308.0310K,2014arXiv1412.8732K}. A more probabilistic approach, based on the notion of the martingale problem, is given by Kulik \cite{1511.00106}.  We should note that
the  construction of semigroups generated by nonlocal integro-differential operators is related to the existence and uniqueness of solutions of stochastic differential equations with jumps. For an overview  of the results  and references  in this direction, including the probabilistic interpretation of the parametrix method  we refer the reader to \cite{2014arXiv1412.8732K}.
We note in passing that in principle \cite{2013arXiv1308.0310K} allows to handle anisotropic kernels,  but
precise upper estimate of the resulting heat kernel
are non-trivial to obtain from the series representation given there.
The reader 
interested in probabilistic methods may consult further results and references in \cite{ MR3022725, 2014arXiv1412.8732K,2013arXiv1308.0310K, MR3544166, 1511.00106, MR3427977}.

Our
development is purely analytic.
We treat operators not manageable by the currently existing methods and give precise estimates for the heat kernels;
our upper bounds of $p_t(x,y)$ are essentially optimal.
We thus give a framework for further investigations of the inhomogeneous Cauchy problem
and of the regularity of solutions to nonlocal equations. The approach also gives guidelines for further developments of the parametrix method. 
In particular, extensions to anisotropic jump kernels $\nu(x,dz)$
with different radial decay profiles, cf. \cite{2014arXiv1403.0912K}, should be possible along the same lines. Such extensions call for estimates and regularity of suitable convolution semigroup majorants, and they are certainly non-trivial.

Here are the main actors of our presentation.
Let $d\in \{1,2,\ldots\}$ and let  $\nu(z,du)\ge 0$ be an integral kernel on $\Rd$ satisfying
 \begin{equation}\label{lev}
\sup_{z\in \Rd}\int_{\Rd\backslash \{0\}}  (1\wedge |u|^2)\, \nu(z,du)<\infty.
\end{equation}
Let $\nu$ be symmetric in the second argument, meaning that for all $z\in\Rd$  and $A\subset \Rd$,
\begin{equation}\label{eq:symm}
\nu(z,A)=\nu(z,-A).
\end{equation}
We note that this is a different symmetry than the one used in the theory of Dirichlet forms \cite{MR2778606}.
If $f:\Rd\to \R$ is a continuous functions vanishing at infinity, then we write $f\in C_0(\Rd)$ and for
$x,z\in \Rd$ and $\delta>0$, we let
\begin{eqnarray*}
L^{z,\delta} f(x):=\frac{1}{2}\int_{|u|>\delta} \big[ f(x+u)+f(x-u)-2f(x)\big] \nu(z,du),
\end{eqnarray*}
and
 \begin{equation}\label{lxd0}
L^z f(x):=\lim_{\delta\to 0} L^{z,\delta}f(x),
\end{equation}
provided a finite limit exits.
We note
$L^z$ and $L^{z,\delta}$ satisfy the
maximum principle: if $f(x_0)=\sup_{x\in \Rd} f(x)$, then   $L^{z,\delta}f(x_0)\le 0$ and $L^zf(x_0)\le 0$.
If, say, $\nu(x_0,du)$ has unbounded support, then we even have  $L^zf(x_0)<0$ provided $f(x_0)=\sup_{x\in \Rd} f(x)>0$, because
$f(x_0+u)+f(x_0-u)$
is close to zero  on a set of positive  measure $\nu(x_0,du)$.
We let
$$
L^{\delta}f(x):=L^{x,\delta}f(x),\qquad Lf(x):= L^xf(x),
$$
and define the domain of $L$:
\begin{equation}\label{dtil}
D(L)= \{ f\in C_0(\rd):\text{ finite $L f(x)$ exists
for all $x\in \rd$}\}.
\end{equation}
We often write $L_x p_t(x,y)$,  etc., meaning that
$L$ acts on the first spatial variable  $x$ of $p_t(x,y)$. 
By the Taylor expansion and \eqref{lev}, $D(L)$ contains $C_0^2(\Rd)$. Here, as usual, $f\in C_0^2(\rd)$ means that
$f$ and all its derivatives of order up to $2$
are continuous and converge to zero at infinity.
We have
\begin{eqnarray}\label{eq:LC2}
&&Lf(x)=\int_{\Rd} \frac12\big[ f(x+u)+f(x-u)-2f(x)\big] \nu(x,du)\\
&&=
\int_{\Rd}\big[ f(x+u)-f(x)-u\cdot  \nabla f(x)\indyk{|u|\leq 1}\big]\nu(x,du),\quad f\in C_0^2(\Rd).\label{eq:wm}
\end{eqnarray}

We now fully specify the properties of
$\nu$ used in this paper. Let $\alpha\in (0,2)$.
Let $\mu_0$ be a symmetric finite measure $\mu_0$ concentrated on the unit sphere $\sphere:=\{x\in \rd: |x|=1\}$. 
We assume that $\mu_0$ is  non-degenerate, i.e.,  not concentrated on a proper linear subspace of $\R$. In particular $\mu_0(\Rd)>0$.
We define
\begin{equation}\label{e:mn}
  \nu_0(A) := \int_{\sphere} \int_0^\infty \indyk{A}(r\theta) r^{-1-\alpha} \,dr \mu_0(d\theta)\,,\quad
  A\subset \Rd\,,
\end{equation}
where $\indyk{A}$ is the indicator function of $A$.
We note that $\nu_0$ is a L\'evy measure, that is,
\begin{equation}\label{eq:V0Lm}
  \int_{\R\backslash \{0\}}(1 \wedge |y|^2)\, \nu_0(dy)<\infty,
\end{equation}
and $\nu_0$ is infinite at the origin, cf.  \eqref{e:mn}.

\begin{definition}
We say that $\nu_0$ is a $\gamma$-measure at $\sphere$ if $\gamma\ge 0$ and
\begin{equation}\label{nu_gamma} 
  \nu_0(B(x,r)) \leq m_0 r^{\gamma}\,,\quad |x|=1\,,\; 0<r<1/2\,.
\end{equation}
\end{definition}
  This is a Haussdorff-type condition on $\nu_0$ outside of the origin.
Since $\nu_0(drd\theta)=r^{-1-\alpha}dr\mu_0(d\theta)$, $\nu_0$ is at least a $1$-measure
and at most a $d$-measure at $\sphere$.
For the rest of the paper
we  fix $\gamma$ and make the following assumptions.
\begin{itemize}
\item[\textbf{A1}.]
$\nu_0$ is given by \eqref{e:mn} with non-degenerate finite symmetric spherical measure $\mu_0$, which is a $\gamma$-measure at $\sphere$ and $\alpha+\gamma>d$.
\item[\textbf{A2}.]\label{A1c}
There exist constants  $M_0>0$, $\eta\in (0,1]$ such that
\begin{equation}\label{M0}
	M_0^{-1}\nu_0(A) \leq \nu(z,A) \leq M_0 \nu_0(A),\quad z\in\Rd,\, A\subset \Rd,
\end{equation}
and
\begin{equation}\label{M00}
  |\nu(z_1,A)-\nu(z_2,A)|\leq M_0 \left( |z_1-z_2|^{\eta}\wedge 1\right) \nu_0(A), \quad z_1,z_2\in \rd, \quad A\subset \rd.
\end{equation}
\end{itemize}
By the Radon-Nikodym theorem, \textbf{A2} is equivalent to having
$\nu(z,du)=h(z,u)\nu_0(du)$, where
$M_0^{-1}\le h(z,u)\leq M_0$ and  $h(z,u)$ is $\eta$-H\"older continuous with respect to $z$.
Note
that \eqref{M0} and \eqref{eq:V0Lm} imply \eqref{lev}.

We now  indicate how to define the heat kernel $p_t(x,t)$ corresponding to $\nu$ (details are given in Section~\ref{para}).
Let $p_t^z(y-x)$ be the the transition probability density corresponding to the L\'evy measure $
\nu(z,\cdot)$, with $z\in \Rd$ fixed, see \eqref{pty}.
For $t>0$, $x,y\in \Rd$ we define the ``zero-order''  approximation of $p_t(x,y)$:
\begin{equation}\label{pt0}
p_t^0(x,y)=
p_t^y(y-x), \quad t>0, \quad x,y\in \rd.
\end{equation}
We note that  it is the ``target point'' $y$ that determines the L\'evy measure $\nu(y,\cdot)$ used to define $p^0_t(x,y)$. This is important for regularity of $x\mapsto p^0_t(x,y)$. We let
\begin{equation}\label{Phi}
\Phi_t(x,y)=\Big(L_x-\prt_t\Big)p_t^0(x,y),
\end{equation}
and \begin{equation}\label{Psi}
\Psi_t(x,y)= \sum_{k=1}^\infty \Phi^{\boxtimes k}_t(x,y),
\end{equation}
We use here the $k$-fold convolution \eqref{conv3}.
Then we let
\begin{equation}\label{r}
p_t(x,y)=p^0_t(x,y)+\big(p^0\boxtimes \Psi\big)_t(x,y).
\end{equation}
The following three theorems reflect the main steps in our development.
\begin{theorem}\label{t-exist} We have
\begin{equation}\label{eq:pofs}
(\partial_t-L_x)p_t(x,y)=0,\qquad t>0,\quad  x,y\in \Rd,
\end{equation}
and for all $f\in C_0(\Rd)$, uniformly in $x\in \rd $ we have
\begin{equation}\label{f-sol}
   \lim_{t\to 0} \int_\Rd f(y) p_t(x,y)\, dy =f(x).
\end{equation}
\end{theorem}
To describe the growth and regularity of
$p_t(x,y)$,
for   $\beta>0$
we let
\begin{equation}\label{g10}
 G^{(\beta)}(x)= (|x|\vee 1)^{-\beta}, \quad G^{(\beta)}_t(x)= \frac{1}{t^{d/\alpha}}G^{(\beta)}\left( \frac{x}{t^{1/\alpha}}\right), \quad t>0,\ x\in \rd.
\end{equation}
Of course, if $\beta>d$, then
\begin{equation}\label{g10b}
 \int_\rd G^{(\beta)}_t(x)dx=\int_\rd G^{(\beta)}(x)dx<\infty, \quad t>0.
\end{equation}

\begin{theorem}\label{t-prop}
There are  constants $C, c>0$  such that
\begin{equation}\label{up100}
 \big|\partial_t^k  p_t(x,y)\big|\leq C t^{-k} e^{ct} G_t^{(\alpha+\gamma)}(y-x), \quad k=0,1,\quad t>0, x,y\in \Rd,
 \end{equation}
 and for all $t>0$, $x_1,x_2,y\in \rd$,
\begin{equation}\label{ptx-hol}
\big| p_t(x_1,y)-p_t(x_2,y)\big| \leq  C \Big(\frac{|x_1-x_2|}{t^{1/\alpha}}\Big)^{\theta}e^{ct} \Big(G_t^{(\alpha+\gamma)}(y-x_1)+
G_t^{(\alpha+\gamma)}(y-x_2)\Big),
\end{equation}
for some  $\theta\in (0, \eta\wedge \alpha\wedge (\alpha+\gamma-d))$. Furthermore, $p_t(x,y)$ is continuous in $y$.
\end{theorem}
The correspondence of $p$ and $L$ is detailed as follows.
 \begin{theorem}\label{t-uni} For $f\in C_0(\Rd)$, $t>0$ and $x\in \rd$ we let
      \begin{equation}\label{t-semi}
P_t f(x)=\int_\rd  p_t(x,y)f(y)dy.
\end{equation}	
Then $(P_t)$ is a strongly continuous Markovian semigroup on $C_0(\Rd)$ and 
the function
 $u(t,x)=P_t f(x)$ defines the unique solution to
the Cauchy problem
\begin{equation}\label{Ca1}
\left\{\begin{split}
\big( \partial_t -L_x\big) u(t,x)&= 0,\quad t>0,\ x\in \Rd,\\
u(0,x)&= f(x),\quad x\in \rd,
\end{split}
\right.
\end{equation}
such that $e^{-\lambda t}u(t,x)
\in C_0([0,\infty)\times \Rd)$ for some $\lambda\in \real$.
\end{theorem}

The structure  of the paper is as follows.
In Section~\ref{sec:prel} we give  notation, definitions and preliminary results, mainly
Lemma~\ref{l:StGenerEst}
and the estimates and regularity of convolution semigroups
whose L\'evy measure
is comparable with $\nu_0$, mainly Lemma~\ref{hol10}.
In Section~\ref{para}
we show that the series defining $p_t(x,y)$ converge
and
we prove Theorem~\ref{t-exist}.
In Section~\ref{time} we estimate the time derivative of $p_t(x,y)$ and we prove Theorem~\ref{t-prop}.
In Section~\ref{Uni}  we prove
Theorem~\ref{t-uni}. We also show that the generator $\mathcal L$ of $(P_t)$ coincides with the operator $L$ on $C^2_0(\rd)$ and that the kernel $p_t(x,y)$ with the above properties is unique.

Before we go to the proofs we discuss typical applications, the sharpness and further questions related to our results.
There exist many measures $\nu_0$ and $\nu$ satisfying the conditions \textbf{A1} and \textbf{A2}.
In fact, $\nu_0$ is a $\gamma$-measure at $\sphere$ if and only if $\mu_0$ is a $(\gamma-1)$-measure at $\sphere$.
We also see that $\nu_0$ is a $d$-measure if and only if it is absolutely continuous with
respect to the Lebesgue measure and has a density function
locally bounded
on $\R\setminus\{0\}$; in this case the condition \textbf{A1} holds trivially. For detailed discussion of this case we refer the reader to \cite{MR1085177} and \cite{MR1220664}. One of possible ways of constructing more general $\nu_0$ is the following.
For every $\gamma\in [1, d]$ there exists a set $F\subset \mathbb{S}$
with positive finite Hausdorff measure of order $(\gamma-1)$ \cite{MR0048540}
and a set $E\subset F$ such that the Hausdorff measure restricted to $E$, say $\mu_0$,  is a nonzero ($\gamma-1$)-measure \cite[Prop. 4.11]{MR2118797}.  Then $\nu_0$ defined by \eqref{e:mn}
is a $\gamma$-measure at $\mathbb{S}$, and
\textbf{A1} holds provided $
\alpha>d-\gamma$.
For instance, if $d=2$ and $E$ is  the usual ternary Cantor set on $\mathbb{S}$
and  $\gamma-1=\log2/\log 3$, then  \textbf{A1} holds provided $
\alpha>1-\log2/\log 3\approx0.3791$.
In the simplest case, if, e.g., the (symmetric non-degenerate) measure $\mu_0$ is a finite sum of Dirac measures, then $\gamma=1$. If we assume $\alpha+1>d=1$ or $2$, then \eqref{up100} with $k=0$  agrees with \cite[(17)]{MR2320691}, which is known to be essentially optimal for $\alpha$-stable convolution semigroups \cite[Theorem~1.1]{MR2286060}. By the same reference, our upper bounds are essentially optimal for general $\gamma$.
Our results and \cite{MR2320691, MR3482695, MR3318251} suggest further questions about more precise estimates of the semigroup in large time, regularity of the resolvent, Harnack inequality for harmonic functions, estimates of the Green function and Poisson kernel, etc.
We also hope that our emphasis on using
the so-called sub-convolution property and auxiliary majorants based on kernels of convolution semigroups, of which many are known at present \cite{MR3165234}, will bring about further progress and more synthetic approach to the Levi's method.

\section{Auxiliary  convolution semigroups }\label{sec:prel}

\subsection{Notation and preliminaries}
Let $\N=\{1,2,\ldots\}$, $\N_0=\{0,1,2,\dots\}$ and $\N_0^d = (\N_0)^d$. For (multiindex) $\beta=(\beta_1,\ldots,\beta_d)\in \N_0^d$ we denote $|\beta|=\beta_1+\ldots+\beta_d$. For $x=(x_1,\ldots,x_d),\,y=(y_1,\ldots,y_d)\in\R$ and $r>0$ we let $x\cdot y=\sum_{i=1}^d x_iy_i$ and $|x|=\sqrt{x\cdot x}$.
We denote by  $B(x,r)\subset \rd$ the ball of radius $r$ centered at $x\in \rd$, so $\sphere=\partial B(0,1)$ is the unit sphere.
All the sets, functions and measures considered in this paper are assumed Borel.
For measure $\lambda$ we let $|\lambda|$ denote the total variation of  $\lambda$.
{\it Constants} mean positive real numbers and we denote them by
$c$, $C$, $c_i$, etc.
For nonnegative functions $f,g$ we write $f\approx g$ to indicate that for some constant $c$,
$c^{-1}f \leq g \leq c f$. We write $c=c(p,q,\dots,r)$ if the constant $c$ can be obtained as a function   of $p,q,\dots, r$ only.

The convolution of measures is, as usual, $\lambda_1*\lambda_2(A)= \int_\rd \lambda_1(A-z)\lambda_2(dz)$, where $A\subset \Rd$.
We also consider the following compositions of functions on space and space-time, respectively:
\begin{align}
(\phi_1\star \phi_2)(x,y)&:= \int_\Rd \phi_1(x,z)\phi_2(z,y)dz,\label{conv2}\\
(f_1\boxtimes f_2) (t,x,y)&:= \int_0^t \int_\Rd f_1(t-\tau,x,z)f_2(\tau, z,y)dzd\tau.
\label{conv3}
\end{align}
Here $x,y\in \Rd$, $t\in [0,\infty)$ and the integrands are assumed to be nonnegative or absolutely integrable.

We keep considering the L\'evy measure $\nu_0$ and the L\'evy  kernel $\nu$ introduced in Section~\ref{sec:iar}.
For clarity, it is always assumed that $A\mapsto \nu(x,A)$ is a Borel measure on $\Rd$ for every $x\in \Rd$ and
$x\mapsto \nu(x,A)$  is Borel measurable for every Borel $A\subset \Rd$.
Since we are about to construct a Feller semigroup corresponding to $\nu$, the assumption of Borel measurability is natural, cf. \cite[Proposition 2.27(f)]{MR3156646}.

By construction, $\nu_0$
is symmetric, non-degenerate and homogeneous of order $-\alpha$:
$$\nu_0(rA)=r^{-\alpha}\nu_0(A),\quad 0<r<\infty, \ A\subset
\Rd.$$
The correspondence of $\nu_0$ and $\mu_0$ is a bijection \cite[Remark~14.4]{MR1739520}.
We call $\mu_0$ the {\it spherical measure} of $\nu_0$.
Since $\mu_0$ is non-degenerate,
\begin{equation}\label{nond}
	\inf_{\xi\in\sphere} \int_{\sphere} |\scalp{\xi}{\theta}|^\alpha\mu_0(d\theta) > 0.
\end{equation}
The respective characteristic (L\'evy-Khintchine) exponent $q_{\nu_0}$ is defined by
\begin{eqnarray}
  q_{\nu_0}(\xi )
  &  =  &  \int_{\Rd\backslash \{0\}}  \big(1-e^{i\scalp{\xi}{u}}+ i\scalp{\xi}{u} \indyk{\{|u|\leq 1\}}\big)\, \nu_0(du)
=   \int_{\Rd\backslash \{0\}}  \big(1-\cos(\scalp{\xi}{u} )\big)\, \nu_0(du)\nonumber\\
&=&\frac{\pi}{2\sin\frac{\pi\alpha}{2}\Gamma(1+\alpha)}
  \int_{\sphere}|\scalp{\xi}{\theta}|^\alpha\, \mu_0(d\theta),\quad \xi \in \Rd\,. \label{DefPhi}
\end{eqnarray}
By scaling and
\eqref{nond},
\begin{equation}\label{eq:Phi_m}
  c_1 |\xi|^\alpha \leq q_{\nu_0}(\xi) \leq c_2 |\xi|^{\alpha},\quad \xi\in\Rd.
\end{equation}
By the L\'evy-Khintchine formula and \eqref{eq:Phi_m} there is a convolution semigroup of probability density functions whose Fourier transform is $\exp(-tq_{\nu_0}(\xi))$, see, e.g.,  \cite{MR0481057, MR1739520}.
If $q_{\nu_0}(\xi)=|\xi|^\alpha$, then the corresponding convolution semigroup $g(t,x)$ satisfies
\begin{equation}\label{eq:itd}
	  g(t,x) \approx t^{-d/\alpha} \wedge \frac{t}{|x|^{d+\alpha}}=G_t^{(d+\alpha)}(x),\quad t>0,\ x\in \Rd\,.
\end{equation}		
The comparison was proved by Blumenthal and Getoor \cite{MR0119247} (see \cite[(29)]{MR3165234} for explicit constants).
In the next section we shall prove variants of the upper bound in \eqref{eq:itd} for the semigroups corresponding to $\nu$. To this end we first learn how to bound integro-differential operators with kernel $\nu$.
In what follows, we denote, as usual,
$$
\diam(A)=\sup\{|x-y|:x,y\in A\}.
$$
We also denote
$$
\delta(A)=\dist(A,0):=\inf\{|x|:x\in A\}.
$$

The lemma below  is an easy consequence of \eqref{e:mn} and \eqref{nu_gamma}.
\begin{lemma} Let
$m_1=\max\{m_0,2^{\gamma}|\mu_0|/\alpha\}$. For every
$A\subset \R$ we have
\begin{equation}\label{eq:nu(A)_est}
	\nu_0(A) \leq m_1\delta(A)^{-\alpha-\gamma}\diam(A)^{\gamma}.
\end{equation}
\end{lemma}
\begin{proof}
If $\delta(A)=0$, then \eqref{eq:nu(A)_est} is trivial, so we assume $\delta(A)>0$.
By the homogeneity of $\nu_0$, for every $x_0\in A$,
\begin{eqnarray*}
	\nu_0(A)
	& \leq & \nu_0(B(x_0,\diam(A))\cap B(0,\delta(A))^c) \\
	&  =   & |x_0|^{-\alpha}\nu_0\left(B\left(\frac{x_0}{|x_0|},\frac{\diam(A)}{|x_0|}\right)\cap B\left(0,\frac{\delta(A)}{|x_0|}\right)^c\, \right).
\end{eqnarray*}
If $\diam(A)/|x_0| \leq \frac{1}{2}$, then from \eqref{nu_gamma} we get
\begin{displaymath}
  \nu_0(A) \leq |x_0|^{-\alpha} m_0 \left(\frac{\diam(A)}{|x_0|}\right)^{\gamma} \leq m_0 \,\delta(A)^{-\alpha-\gamma}\diam(A)^{\gamma},
\end{displaymath}
and if $\diam(A)/|x_0| \geq \frac{1}{2}$, then
\begin{eqnarray*}
	\nu_0(A)
	& \leq & |x_0|^{-\alpha} \nu_0\left(B\left(0,\frac{\delta(A)}{|x_0|}\right)^c\, \right)
	  =  |x_0|^{-\alpha} \frac{|\mu_0|}{\alpha} \left(\frac{\delta(A)}{|x_0|}\right)^{-\alpha} \nonumber \\
	& = & \frac{| \mu_0|}{\alpha} \delta(A)^{-\alpha} \leq \frac{2^{\gamma}| \mu_0|}{\alpha} \delta(A)^{-\alpha-\gamma}\diam(A)^{\gamma}.
\end{eqnarray*}
Thus, in either case we get \eqref{eq:nu(A)_est}.
\end{proof}

Let $C^2_b(\rd)$ be the class of all the functions bounded together with their derivatives of order up to $2$.
For $t>0$, $x\in \rd$ and $f\in C_b^2(\Rd)$ we denote
\begin{displaymath}
	\modgener f (x) := \int_{\rd\backslash \{0\}} \left| f(x+u)-f(x)-\scalp{u}{\nabla_x f(x)}\indyk{ \{|u|\leq t^{1/\alpha}\} } 	\right|\, \nu_0(du).
\end{displaymath}
The lemma below is the  main result of this  subsection.
\begin{lemma}\label{l:StGenerEst}
  Let $f:(0,\infty)\times \Rd \to \RR$ be  such that $f_t(\cdot):= f(t,\cdot)\in C^2_b(\Rd)$
  for every $t>0$ and there are constants $K,\zeta>0,\kappa\in [0,\alpha+\gamma-d)$ such that
  \begin{equation}\label{eq:avfest}
	  | \partial^\beta_x  f_t(x)| \leq  K t^{-(\zeta+|\beta|)/\alpha} (1+t^{-1/\alpha}|x|)^{-\gamma-\alpha+\kappa}, \quad x\in\Rd,\, t>0,
  \end{equation}
for every  multiindex $\beta\in \N_0^d$ with $|\beta|=0$ or $2$.
Then $c_{\mathcal{A}}$ exists such that
  \begin{displaymath}
	  \modgener f_t(x)
	  \leq  c_{\mathcal{A}}K  t^{-1-\zeta/\alpha} (1+t^{-1/\alpha}|x|)^{-\gamma-\alpha+\kappa}, \quad x\in\Rd,\, t>0.
  \end{displaymath}
\end{lemma}
\begin{proof}
  We have $ \modgener f_t(x)=I_1 + I_2$, where
  \begin{eqnarray*}
	 I_1
	  &  =  &  \int_{|u|\leq t^{1/\alpha}}
	           \left|f_t(x+u) - f_t(x)-\scalp{u}{\nabla_x f_t(x)} \indyk{\{|u|\leq t^{1/\alpha}\}} \right|\, \nu_0(du), \\
	  I_2&  =  & \int_{|u|> t^{1/\alpha}}
	           \left|f_t(x+u) - f_t(x)-\scalp{u}{\nabla_x f_t(x)} \indyk{\{|u|\leq t^{1/\alpha}\}} \right|\, \nu_0(du).
  \end{eqnarray*}
  From the Taylor expansion and \eqref{eq:avfest} we get
  \begin{eqnarray*}
    I_1
    &  =   & \int_{|u|\leq t^{1/\alpha}} \left|f_t(x+u) - f_t(x) - \scalp{u}{\nabla_x f_t(x)} \right|\,   \nu_0(du)  \\
    & \leq & Kd^2 2^{\alpha+\gamma-\kappa-1} \int_{|u|\leq t^{1/\alpha}} |u|^2 t^{-(\zeta+2)/\alpha} (1+t^{-1/\alpha}|x|)^{-\gamma-\alpha+\kappa} \, \nu_0(du) \\
    &  =   & Kd^2 2^{\alpha+\gamma-\kappa-1} t^{-(\zeta+2)/\alpha} (1+t^{-1/\alpha}|x|)^{-\gamma-\alpha+\kappa} \frac{|\mu_0|}{2-\alpha} t^{(2-\alpha)/\alpha} \\
    &  =   & Kc_1 t^{-1-\zeta/\alpha} (1+t^{-1/\alpha}|x|)^{-\gamma-\alpha+\kappa}.
  \end{eqnarray*}
 We split $I_2$ in the following way,
  \begin{eqnarray*}
   I_2&=  &
         \int_{|u| > t^{1/\alpha}} \left| f_t(x-u) - f_t(x) \right|
             \, \nu_0(du) \\
    & \leq & \int_{|u| > t^{1/\alpha}} |f_t(x-u)| \, \nu_0(du) + |f_t(x)|  \int_{|u| > t^{1/\alpha}}\, \nu_0(du) \\
    &  =   & \left( \int_{|u| > t^{1/\alpha},\, |x - u|>t^{1/\alpha}}
             +  \int_{|u| > t^{1/\alpha},\, |x - u|\leq t^{1/\alpha}} \right) \left| f_t(x-u) \right| \, \nu_0(du)
    + \, \left| f_t(x) \right| \frac{|\mu_0|}{\alpha t} \\
    &  = : &  I_{21} + I_{22} + |f_t(x)|  \frac{|\mu_0|}{\alpha t}.
  \end{eqnarray*}
  Using \eqref{eq:nu(A)_est} and \eqref{eq:avfest} we obtain
  \begin{eqnarray*}
    I_{22}
    & \leq & K \int_{|u| > t^{1/\alpha},\, |x - u|\leq t^{1/\alpha}}
               t^{-\zeta/\alpha} \, \nu_0(du) \\
    &   =  & K t^{-\zeta/\alpha} \nu_0 \left( B(x,t^{1/\alpha})\cap B(0,t^{1/\alpha})^c \right) \\
		& \leq & Km_1 t^{-\zeta/\alpha} (2t^{1/\alpha})^{\gamma} \left(\max\{|x|-t^{1/\alpha},t^{1/\alpha}\}\right)^{-\gamma-\alpha} \\
    & \leq & Kc_2 t^{-1-\zeta/\alpha}\left( 1+t^{-1/\alpha}|x| \right)^{-\gamma-\alpha}.
  \end{eqnarray*}
  In order to estimate $I_{21}$ we define
  \begin{eqnarray*}
    J_1
    &   =  & \int_{|u| > t^{1/\alpha},\, \max\{|x|/4,t^{1/\alpha}\} > |x - u|>t^{1/\alpha}} \left| f_t(x-u) \right| \, \nu_0(du),\\
    J_2
    &   =  & \int_{|u| > t^{1/\alpha},\, |x - u|\geq \max\{|x|/4,t^{1/\alpha}\}}
             \left| f_t(x-u) \right| \, \nu_0(du),
  \end{eqnarray*}
  and observe  that $I_{21}=J_1+J_2$. Using \eqref{eq:avfest} we get
  \begin{eqnarray*}
    J_2
    & \leq & \int_{|u| > t^{1/\alpha},\, |x - u|\geq |x|/4} K t^{-\zeta/\alpha} (1+t^{-1/\alpha}|x-u|)^{-\gamma-\alpha+\kappa} \, \nu_0(du) \\
    & \leq & (K |\mu_0|/\alpha) t^{-1-\zeta/\alpha} (1+t^{-1/\alpha}|x|/4)^{-\gamma-\alpha+\kappa} \\
		& \leq & K c_3 t^{-1-\zeta/\alpha} (1+t^{-1/\alpha} |x|)^{-\gamma-\alpha+\kappa}.
  \end{eqnarray*}
If $|x|<4t^{1/\alpha}$, then $J_1=0$. If $|x|\geq 4t^{1/\alpha}$, then $L:=\lfloor \log_2 (t^{-1/\alpha}|x|/4) \rfloor\ge 0$, and
  \begin{eqnarray*}
    J_1
    & \leq & \int_{|u| > t^{1/\alpha},\, |x|/4 > |x - u|>t^{1/\alpha}} K t^{-\zeta/\alpha} (1+t^{-1/\alpha}|x-u|)^{-\gamma-\alpha+\kappa}
             \, \nu_0(du) \\
    & \leq & \sum_{n=0}^{L} \, \int_{2^{n+1}t^{1/\alpha}\geq |x-u| > 2^nt^{1/\alpha}}
             K t^{-\zeta/\alpha} (1+t^{-1/\alpha}|x-u|)^{-\gamma-\alpha+\kappa} \, \nu_0(du) \\
    & \leq & K t^{-\zeta/\alpha} \sum_{n=0}^{L} 2^{-n(\alpha+\gamma-\kappa)} \nu_0\left( B(x,2^{n+1}t^{1/\alpha}) \right) \\
    & \leq & K t^{-\zeta/\alpha} \sum_{n=0}^{L} 2^{-n(\alpha+\gamma-\kappa)} m_1 2^{\alpha+\gamma} |x|^{-\gamma-\alpha} (2^{n+2}t^{1/\alpha})^\gamma \\
 		& \leq & K c_4 t^{(-\zeta+\gamma)/\alpha} |x|^{-\gamma-\alpha}
     \leq  K c_5 t^{-1-\zeta/\alpha}(1+t^{-1/\alpha}|x|)^{-\gamma-\alpha},
  \end{eqnarray*}
  where in the fourth inequality we use \eqref{eq:nu(A)_est} and the fact that $2^{n+1}t^{1/\alpha}\leq |x|/2$ for $n\leq L$.
  We obtain
  \begin{displaymath}
	  I_{2} = I_{21} + I_{22} + |f_t(x)|  \frac{|\mu_0|}{\alpha t}
	  \leq  K c_6 t^{-1-\zeta/\alpha} (1+t^{-1/\alpha}|x|)^{-\gamma-\alpha+\kappa},
  \end{displaymath}
  and the lemma follows.
\end{proof}

\subsection{Estimates of $p_t^z(x)$}

In this section we  estimate the convolution semigroup corresponding to the L\'evy measure $\nu(z,\cdot)$ with fixed but arbitrary $z\in \rd$. We are interested in majorants which are integrable in space, like \eqref{g10b}.
We note that the results \cite{MR2320691} cannot be directly used here because we also need  H\"older continuity of  $z\mapsto p_t^z$, which is crucial for the proof of Theorem~\ref{t-exist}. 
Recall that each $\nu(z,\cdot)$ is symmetric and comparable to $\nu_0$, i.e. it satisfies \eqref{M0}.
Therefore,
\begin{equation}\label{qu1}
  q(z,\xi) :=\int_{\Rd\backslash\{0\}} (1-e^{i\xi \cdot u}+i\xi \cdot u \I_{\{|u|\leq 1\}})\, \nu(z,du)=
  \int_{\Rd\backslash\{0\}} (1-\cos \xi \cdot u) \, \nu(z,du),
\end{equation}
is real-valued and  there exist constants $c,C>0$ such that
\begin{equation}\label{growth}
c |\xi|^\alpha \leq q(z,\xi) \leq C|\xi|^\alpha, \quad \xi, z\in\Rd.
\end{equation}
By \eqref{growth},
\begin{equation}\label{pty}
p_{t}^z(x):= (2\pi)^{-d} \int_\Rd e^{-ix\cdot \xi -tq(z,\xi)}\, d\xi,\quad t>0, \, x\in \rd,
\end{equation}
is infinitely smooth in $t$ and $x$. Note that  for each $z$, $(p_t^z)_{t>0}$ is a convolution semigroup of probability densities.
The operator $L^z$ equals the generator of the semigroup on $C^2_0(\Rd)$, see, e.g., \cite{MR3514392}. Therefore,
\begin{equation}\label{Couchy_for_Levy}
  \partial_t p_{t}^z(x) = L^z p_t^z (x),\quad t>0,\ x\in\rd,
\end{equation}
and so
$$
  \partial_t p_t^z(x) = \int_\rd \left( p_t^z(x+u) - p_t^z(x) - \scalp{u}{\nabla_x p_t^z(x)} \indyk{\{|u|\leq 1\}}  \right) \, \nu(z,du).
$$
We recall  the definition \eqref{g10} and give an approximation for convolutions of $G^{(\beta)}_t$.
\begin{lemma}\label{subconvG}
  For every $\beta\in (d,d+2)$,
  \begin{equation}\label{convG}
	  \int_{\rd} G_{t-s}^{(\beta)}(x-z)G_s^{(\beta)}(z)\,dz \approx G_t^{(\beta)}(x),\quad x\in\Rd,\ 0<s<t.
	\end{equation}
\end{lemma}
\begin{proof}
Denote $\delta=\beta-d$.
  Let $g(t,x)$
	be the density function of the isotropic rotation invariant $\delta$-stable L\'evy process. 
	We have $G_t^{(\beta)}(x) \approx g(t^{\delta/\alpha},x)$, see, e.g., \cite{MR0119247}.
	\begin{eqnarray*}
	  \int_{\rd} G_{t-s}^{(\beta)}(x-z)G_s^{(\beta)}(z)\,dz
		& \approx & \int_{\rd} g((t-s)^{\delta/\alpha},x-z)g(s^{\delta/\alpha},z)\,dz \\
		&	  =    & g((t-s)^{\delta/\alpha}+s^{\delta/\alpha},x).
	\end{eqnarray*}
	Since
	$$
	  t^{\delta/\alpha} \leq (t-s)^{\delta/\alpha} + s^{\delta/\alpha} \leq 2 t^{\delta/\alpha},
	$$
	from \eqref{eq:itd} we get
	$$
	  g((t-s)^{\delta/\alpha}+s^{\delta/\alpha},x) \approx
		t^{-d/\alpha} \left(1\vee \frac{|x|}{t^{1/\alpha}}\right)^{-\beta}
		= G_t^{(\beta)}(x),
	$$
	and \eqref{convG} follows. 	
\end{proof}
\begin{lemma}\label{th-ptx2}
For every $\beta\in\N_0^d$ there is $c=c(\nu_0,\beta, M_0)>0$
such that
$$
  |\partial^\beta_x p_t^z(x)| \leq c t^{-|\beta|/\alpha}  G_t^{(\alpha+\gamma)}(x) ,\quad t>0,\, x,z \in\Rd.
$$
\end{lemma}
The proof of Lemma~\ref{th-ptx2} relies on auxiliary results which we give  first.
Fix an arbitrary
$z\in\Rd$. Let $\vartheta(\cdot)=\nu(z,\cdot).$
For $\varepsilon>0$ let $\bar{\vartheta}_{\varepsilon}=\indyk{B(0,\varepsilon)^c}\vartheta$,
$\tilde{\vartheta}_\varepsilon=\indyk{B(0,\varepsilon)}\vartheta$, and
$$
  q_{\bar{\vartheta}_{\varepsilon}} (\xi) = \int_\rd \big(1-e^{i\xi \cdot u}\big)\, \bar{\vartheta}_{\varepsilon}(du),
  \quad q_{\tilde{\vartheta}_\varepsilon}(\xi) = \int_\rd \big(1-e^{i\scalp{\xi}{u}}+i\scalp{\xi}{u} \indyk{\{|u|\leq 1\}}\big)\, \tilde{\vartheta}_\varepsilon(du),
$$
and
$$
  \tilde{p}^\varepsilon_t(x)=\big( \mathcal{F}^{-1} \exp(-tq_{\tilde{\vartheta}_\varepsilon}(\cdot))\big)(x), \quad t>0,\, x \in\Rd,
$$
where  $\mathcal{F}^{-1}$ is the inverse Fourier transform.
By \eqref{eq:Phi_m} we see that $\tilde{p}^\varepsilon_t(x)$ is smooth.
Further,  the  probability measure with the characteristic function  $\exp(-tq_{\bar{\vartheta}_{\varepsilon}} (\xi))$ is
\begin{equation}\label{eexp}
\bar{P}^\varepsilon_t(dy)=e^{-t|\bar{v}_\varepsilon|}
    \sum_{n=0}^\infty
    \frac{t^n\bar{v}_\varepsilon^{*n}(dy)}{n!}\,,\quad t>0.
\end{equation}
We have
\begin{equation}\label{e:wpt}
  p_t^z=
  \tilde{p}^\varepsilon_t
  *
  \bar{P}^\varepsilon_t
  \,.
\end{equation}
The first step in the proof of Lemma~\ref{th-ptx2} is to estimate the terms in the series (\ref{eexp}).
The following Lemma is a version of \cite[Lemma ~1]{MR2320691} and \cite[Cor.~10]{MR3357585}.
\begin{proposition}\label{Conv_est}
There exists
$m_3>0$ such that for $\varepsilon>0$ and $n\geq 1$,
$$
  \bar{\vartheta}_{\varepsilon}^{*n}(B(x,r)) \leq
   m_3^n  \varepsilon^{-(n-1)\alpha}|x|^{-\alpha-\gamma} r^\gamma\,,\quad x\in\R\setminus\{ 0 \},\, r<|x|/2.
$$
Consequently,
$$
	\bar{P_t^\varepsilon}\left(B(x,r)\right) \leq \varepsilon^\alpha e^{m_3\varepsilon^{-\alpha}t} |x|^{-\alpha-\gamma} r^\gamma.
$$
\end{proposition}
\begin{proof}
The result follows from  \cite[Lemma 9]{MR3357585}. Indeed, one can check that the conditions  (23) from  \cite[Lemma 9]{MR3357585} hold true with $f(s)=s^{-\alpha-\gamma}$, which gives
\begin{equation}\label{eq:nu_n*_est}
     \bar{\vartheta}_{\varepsilon}^{*n}(A) \leq c^n \mathfrak{q}_{\nu_{0}}(1/\varepsilon)^{n-1}
    f\left(\delta(A)/2\right)\diam(A)^{\gamma},
\end{equation}
where
$ \mathfrak{q}_{\nu_0}(r)=\sup_{|\xi|\leq r}  q_{\nu_0}(\xi)$, $ r>0$.
Since $\mathfrak{q}_{\nu_0}(r)\approx r^\alpha$, we get the required estimates.  \qedhere
\end{proof}

We denote $\bar{P}_t=\bar{P}_t^{t^{1/\alpha}}$,
$\tilde{p}_t=\tilde{p}_t^{t^{1/\alpha}}$ and $\tilde{\vartheta}=\tilde{\vartheta}_{t^{1/\alpha}}$.

\begin{lemma}\label{Cor:tildepest}
For every $n\in \N_0$ and  $\beta\in\N_0^d$ there is
$c>0$ 
such that
\begin{equation}\label{e:cest1}
  |\partial^\beta_x \tilde{p}_t(x)| \leq c t^{-|\beta|/\alpha}  G_t^{(n)}(x),\quad t>0,x\in\Rd.
\end{equation}
\end{lemma}
\begin{proof}
  Let $g_t(x) = t^{d/\alpha} \tilde{p}_t(t^{1/\alpha}x)$, $t>0$, $x\in \Rd$. For each $t$, $g_t(x)$ is the density function of an infinitely divisible distribution. We denote by    $\phi_t(\xi)$ and $\eta_t(du)$ the corresponding characteristic exponent and L\'evy measure, respectively.
   To prove \eqref{e:cest1} we will apply  \cite[Prop. 2.1]{MR2995789}, for which it suffices to check that
  \begin{equation}\label{cond1}
  \int_\rd |\xi|^k e^{- \RRR \phi_t(\xi)}d\xi< c, \quad \int_\rd |u|^{k
  } \eta_t(du)<c
  \end{equation}
for every $k\geq 2$ with constant $c$ independent of $t$. Indeed, a direct calculation gives $\eta_t(A)=t\tilde{\vartheta}_{t^{1/\alpha}}(t^{1/\alpha}A)$. Then
by \eqref{M0} we get
$$
	\int |y|^k\, \eta_t(dy) \leq M_0 t \int_{|y|<t^{1/\alpha}} \left(\frac{|y|}{t^{1/\alpha}}\right)^k\, \nu_0(dy)  = \frac{M_0 |\mu_0|}{k-\alpha}.
$$
Further,
\begin{eqnarray*}
	\RRR \phi_t (\xi)
	&  =   & \int \left( 1 - \cos \left( \scalp{\xi}{y} \right)\right)\, \eta_t(dy)
	 \geq  M_0^{-1} t \int_{|y|<t^{1/\alpha}} \left( 1 - \cos \left( \scalp{\xi}{\frac{y}{t^{1/\alpha}}} \right)\right)\, \nu_0(dy) \\
	&  =   & M_0^{-1} t q_{\nu_0}(\xi/t^{1/\alpha})
	         - M_0^{-1} t \int_{|y|\geq t^{1/\alpha}} \left( 1 - \cos \left( \scalp{\xi}{\frac{y}{t^{1/\alpha}}} \right)\right)\, \nu_0(dy) \\
	& \geq & M_0^{-1} t q_{\nu_0}(\xi/t^{1/\alpha}) - M_0^{-1} t \nu_0(B(0,t^{1/\alpha})^c) \geq c_1 |\xi|^\alpha - c_2.
\end{eqnarray*}
Therefore,
\begin{equation*}
  \int e^{-\RRR \phi_t(\xi)}|\xi|^k \, d\xi
 \leq  e^{c_2} \int e^{-c_1|\xi|^{\alpha}}|\xi|^k\, d\xi
 \leq  c_3 < \infty.
\end{equation*}
Thus, \eqref{cond1} holds true and applying the result from  \cite[Prop.2.1]{MR2995789} we get
\begin{displaymath}
	|\partial^\beta_x g_t(x)| \leq c_4 \left(1+|x| \right)^{-n},\quad n\ge 0,\, t>0, \, x\in \rd.
\end{displaymath}
 Coming back to $\tilde{p}_t$ we get the desired estimate.
\end{proof}

\begin{proof}[Proof of Lemma~\ref{th-ptx2}]
   We have
 \begin{eqnarray*}
    |\partial^\beta_x p_t^z(x) |
  &   =  & | \left(2\pi\right)^{-d} (-i)^{|\beta|} \int \xi^\beta e^{-i\scalp{x}{\xi}} e^{-tq_\vartheta(\xi)}\, d\xi |\\
  & \leq & \left(2\pi\right)^{-d} \int |\xi|^{|\beta|} e^{-t c_1 |\xi|^\alpha} \, d\xi
     =   c_2 t^{(-d-|\beta|)/\alpha},\quad t>0, \, x\in\Rd.
 \end{eqnarray*}
 Using Proposition~\ref{Conv_est}  and Lemma~\ref{Cor:tildepest} with $n\geq \alpha+\gamma$, for $|x|>2t^{1/\alpha}$ we obtain
  \begin{eqnarray*}
    &&\left|\partial_x^\beta\left(\tilde{p}_t \ast \bar{P}_t \right)(x)\right|
       =   \left| \int_\rd \partial_x^\beta \tilde{p}_t(x-y) \bar{P}_t(dy) \right|
     \leq \int_\rd \left| \partial_x^\beta \tilde{p}_t(x-y) \right| \bar{P}_t(dy) \\
    && \leq  c_{3}t^{ \frac{-d-|\beta|}{\alpha} } \int_\rd (1+t^{-1/\alpha}|x-y|)^{-n} \bar{P}_t(dy)\\
    &&  =   c_{3}t^{ \frac{-d-|\beta|}{\alpha} } \int_\rd \int_0^{(1+t^{-1/\alpha}|x-y|)^{-n}} \,ds\, \bar{P}_t(dy)\\
    && =   c_{3}t^{ \frac{-d-|\beta|}{\alpha} } \int_0^1 \int_\rd
             \indyk{ (1+t^{-1/\alpha}|x-y|)^{-n}>s} \, \bar{P}_t(dy) ds \\
    &&   =   c_{3} t^{ \frac{-d-|\beta|}{\alpha} } \int_0^1 \bar{P}_t\left(B(x,t^{1/\alpha}(s^{-\frac{1}{n}}-1))\right) ds,
    \end{eqnarray*}
    thus
  \begin{eqnarray*}
    &&\left|\partial_x^\beta\left(\tilde{p}_t \ast \bar{P}_t \right)(x)\right| \\
    && \leq  c_4 t^{ \frac{-d-|\beta|}{\alpha} }  \Big(\int_{(1+\frac{|x|}{2t^{1/\alpha}})^{-n}}^1 t |x|^{-\alpha-\gamma}
              \Big(t^{1/\alpha}(s^{-\frac{1}{n}}-1) \Big)^\gamma\, ds
             + \int_0^{(1+\frac{|x|}{2t^{1/\alpha}})^{-n}}\, ds \Big) \\
    && \leq  c_4 t^{ \frac{-d-|\beta|}{\alpha} } \Big( t^{1+\gamma/\alpha} |x|^{-\alpha-\gamma} \int_0^1
             s^{-\gamma/n}\, ds +  \Big(1+\frac{|x|}{2t^{1/\alpha}} \Big)^{-n}  \Big) \\
    &&  =    c_5 t^{ \frac{-d-|\beta|}{\alpha} }  \Big( t^{1+\gamma/\alpha} |x|^{-\alpha-\gamma}
             + \Big(1+\frac{|x|}{2t^{1/\alpha}}\Big)^{-n}  \Big)
              \leq  c_6 t^{ \frac{-d-|\beta|}{\alpha} } \Big(1+\frac{|x|}{t^{1/\alpha}} \Big)^{-\alpha-\gamma}. \qedhere
  \end{eqnarray*}
\end{proof}
For the regularity in time we have another estimate. Here the spatial bound is satisfactory, cf. \eqref{g10b}, and the temporal growth at $t=0$ will later be tempered by making use of cancellations.
\begin{lemma}\label{l:dertxpest}
  For every $\beta\in\N_0^d$ there exists a constant $c>0$ such that
  \begin{equation}\label{ptx-der}
	  |\partial_t \partial^\beta_x p_t^z (x)| \leq  c t^{-1-|\beta|/\alpha}  G_t^{(\alpha+\gamma)}(x),\quad x,z\in\Rd,t>0.
\end{equation}
\end{lemma}
\begin{proof}
 It follows from \eqref{pty} and \eqref{M0} that
  \begin{displaymath}
	   \partial_t \partial^\beta_x p_t^z (x)  =   \partial^\beta_x \partial_t p_t^z (x) = (2\pi)^{-d} \int_\Rd q(z,\xi) (-1)^{|\beta|+1} \xi^\beta e^{-ix\cdot \xi -tq(z,\xi)}\, d\xi.
  \end{displaymath}
  Recall that
  \begin{displaymath}
	  \partial_t p_t^z(x) = \int_\rd \left( p_t^z(x+u) - p_t^z(x) - \scalp{u}{\nabla_x p_t^z(x)} \indyk{\{|u|\leq 1\}}  \right) \, \nu(z,du),
  \end{displaymath}
 cf. \eqref{Couchy_for_Levy}.
  Differentiating with respect to $x$ and   using \textbf{A2} we get
  \begin{eqnarray*}
	  \left| \partial_t \partial^\beta_x p_t^z (x) \right|
	  &  =   & \left| \partial^\beta_x \partial_t p_t^z (x) \right| \\
	  &  =   & \Big| \int \left( \partial^\beta_x p_t^z (x+u) - \partial^\beta_x p_t^z(x) - \partial^{\beta}_x \scalp{u}{\nabla_x  p_t^z(x)}
	           \indyk{\{|u|\leq 1 \}} \right)
	           \, \nu(z,du) \Big| \\
	  & \leq & C \int \left| \partial^\beta_x p_t^z (x+u) - \partial^\beta_x p_t^z(x) - \scalp{u}{ \nabla_x \partial^{\beta}_x p_t^z(x)}\indyk{\{ |u|\leq t^{1/\alpha} \}} \right| \, \nu_0 (du).
	\end{eqnarray*}
 Using  Proposition~\ref{th-ptx2} and Lemma~\ref{l:StGenerEst} we obtain \eqref{ptx-der}.
\end{proof}

The H\"older continuity of  $z\mapsto p_t^z$,
will be proved in Lemma \ref{hol10} after some auxiliary lemmas. In the first one
we record the symmetry of the operators $L^w$.
\begin{lemma}\label{Lw_symmetry}
 For every $w\in\Rd$ the operator $L^w$ is symmetric, i.e.,
  \begin{equation}\label{eq:symmofL}
	  \int_{\Rd} \varphi(x) L^w f (x)\, dx = \int_{\Rd} L^w \varphi(x) f(x)\, dx
	\end{equation}
	for all $\varphi,f\in C^2_b(\Rd)\cap L^1(\Rd)$.
\end{lemma}
\begin{proof} For every $\delta>0$ we have
$$
\int\int_{|u|>\delta} |f(x+u)\varphi(x)|\nu(w,du) dx \leq \|f\|_\infty \nu(w,B(0,\delta)^c) \|\varphi \|_1 < \infty.
$$
Hence,  by Fubini's theorem, change of variables and symmetry of $\nu(w,\cdot)$ we get
  \begin{eqnarray*}
	  \int_{\Rd} \varphi(x) \int_{|u|>\delta} f (x+u) \nu(w,du) \, dx
		&  =  & \int_{\Rd} f(y) \int_{|u|>\delta} \varphi(y+u) \, \nu(w,du)\, dy.
	\end{eqnarray*}
	By subtracting
	$
	  \int_{\Rd} \int_{|u|>\delta} f(x)\varphi(x) \nu(w,du) dx,
	$
	we obtain $$\int_{\Rd} \varphi(x) L^{w,\delta} f (x) \, dx = \int_{\Rd} L^{w,\delta} \varphi(x) f(x)\, dx.$$
	Let $\delta \to 0$. By dominated convergence we get \eqref{eq:symmofL}, since for  $g\in C^2_b(\Rd)$, $\delta\in (0,1)$,
	\begin{eqnarray*}
	  |L^{w,\delta} g (x)|
		&   =  & \left|\int_{|u|>\delta} \left( g(x+u)-g(x)-\indyk{B(0,1)}(u)\scalp{\nabla g (x)}{u}\right)\, \nu(w,du) \right| \\
		& \leq & \tfrac{1}{2} d^2\|g\|_2 \int_{|u|<1} |u|^2\, \nu(w,du) + 2 \| g\|_\infty \int_{|u|>1}\, \nu(w,du) < \infty.
	\end{eqnarray*}
	Here, as usual, $\|g\|_2=\sup_{x\in\Rd, \beta\in\N_0^2} |\partial^\beta g(x)|$.
\end{proof}

\begin{corollary}\label{Lanihil}
For every $f\in C^2_b(\Rd)\cap L^1(\Rd)$ and $w\in\Rd$ such that $L^w f \in L^1(\Rd)$,
$$
  \int_{\Rd} L^w f(x)\, dx = 0.
$$
\end{corollary}
\begin{proof}
  Let $\varphi_n \in C^\infty_c(\Rd)$ be such that $0\leq\varphi_n(x)\leq 1$ and $\varphi_n(x) = 1$ for every $x\in B(0,n)$
	and $\|\varphi_n \|_2 \leq c_0$ for every $n\in\N$. Note that
$$
 |L^w \varphi_n(x)| \leq c_1 \int_{|u|<1} |u|^2\, \nu(w,du) + 2 \int_{|u|\geq 1}\, \nu(w,du) < \infty,
$$
	and for $n > |x|$
	we have
	$$
|L^w \varphi_n(x)| = | \int_{|x+u|> n} \left( \varphi_n(x+u)-1 \right)\, \nu(w,du) | \leq 2 \nu(w,B(-x,n)^c),
$$
	which yields $\lim_{n\to \infty} L^w \varphi_n(x) = 0$ for every $x\in\Rd$.
	By the symmetry of $L^w$,
	$$
	  \int_{\Rd} \varphi_n(x) L^w f (x)\, dx = \int_{\Rd} L^w \varphi_n(x) f(x)\, dx,
	$$
	and the corollary follows by the dominated convergence  theorem.
	\end{proof}
	
\begin{lemma}\label{phi_properties}
  Let $t>0$, $x,w_1,w_2\in\rd$ and
	$$
  \phi(s) = \left\{
  \begin{array}{ccc}
    p^{w_1}_{t-s}*p^{w_2}_s(x)& \mbox{  if  } & s\in (0,t),\\
    p_t^{w_1}(x) & \mbox{  if  } & s=0, \\
		p_t^{w_2}(x) & \mbox{  if  } & s=t.
  \end{array}\right.
  $$
	Then $\phi$ is continuous on $[0,t]$, $\partial_s \phi(s)$ exists on $(0,t)$ and
$$
  \partial_s \phi(s) = \int_{\rd} \left( p_{t-s}^{w_1}(z-x) L^{w_2} p_s^{w_2}(z)-
   p_s^{w_2}(z)L^{w_1}p_{t-s}^{w_1}(z-x)\right)\, dz.
$$
\end{lemma}
\begin{proof}
  We have
	\begin{eqnarray*}
	  |\phi(s)-\phi(0)|
		&  =   & \left| \int_{\rd} p_{t-s}^{w_1}(x-z)p_s^{w_2}(z)\, dz - p_{t}^{w_1}(x)\right| \\
		& \leq & \left| \int_{\rd}  p_{t}^{w_1}(x-z)p_s^{w_2}(z)\, dz - p_{t}^{w_1}(x) \right| \\
		&       & +
		         \int_{\rd} |p_{t-s}^{w_1}(x-z) - p_{t}^{w_1}(x-z)| p_s^{w_2}(z)\, dz \\
		& = & I_1(s) + I_2(s),
	\end{eqnarray*}
	and $\lim_{s\to 0} I_1(s) = 0$, since the semigroup $P^{w_2}_s f(x)= \int_{\rd} f(x-z) p_s^{w_2}(z)\, dz$ is strongly
	continuous and $p_{t}^{w_1}(x-\cdot)\in C_0(\rd)$.
	For $s\in (0,t/2)$ from Lemma \ref{l:dertxpest} we get
	\begin{eqnarray*}
	  &&|p_{t-s}^{w_1}(x-z) - p_{t}^{w_1}(x-z)|
		 \leq  s \! \sup_{u\in (t-s,t)}|\partial_u p_{u}^{w_1}(x-z)|  \\
		&& \leq  c s \sup_{u\in (t-s,t)} \left\{ u^{-1} G_u^{(\alpha+\gamma)}(x-z)\right\}
		\leq  c st^{-1-d/\alpha} \left(\frac{|x-z|}{t^{1/\alpha}}\vee 1\right)^{-\alpha-\gamma}.
	\end{eqnarray*}
From the strong continuity of $s\mapsto P^{w_2}_s$ we have
	$$
	  \lim_{s\to 0} \int_{\rd} \left(\frac{|x-z|}{t^{1/\alpha}}\vee 1\right)^{-\alpha-\gamma} \!\!\!p_s^{w_2}(z)\, dz =
		\left(\frac{|x|}{t^{1/\alpha}}\vee 1\right)^{-\alpha-\gamma},
	$$
therefore $\lim_{s\to 0} I_2(s) = 0$.
	 This yields the continuity of $\phi$ at $s=0$. The proof
	of the continuity at $s=t$ is analogous.

	For every $z\in\rd$ and $s\in (0,t)$ from \eqref{Couchy_for_Levy} we obtain
	$$
	  \partial_s \left( p_{t-s}^{w_1}(x-z)p_s^{w_2}(z) \right) = p_{t-s}^{w_1}(z-x) L^{w_2} p_s^{w_2}(z)-
   p_s^{w_2}(z)L^{w_1}p_{t-s}^{w_1}(z-x).
	$$
	From Lemma \ref{th-ptx2} and Lemma \ref{l:StGenerEst} we get
	$$
	  |\partial_s \left( p_{t-s}^{w_1}(x-z)p_s^{w_2}(z) \right)| \leq c G_s^{(\alpha+\gamma)}(z)  G_{t-s}^{(\alpha+\gamma)}(z-x) \left( s^{-1} + (t-s)^{-1} \right),
	$$
	hence for every $\delta\in (0, t/2)$ and $s\in (\delta,t-\delta)$ we have
	\begin{equation}\label{pse2}
	  |\partial_s \left( p_{t-s}^{w_1}(x-z)p_s^{w_2}(z) \right)| \leq
		2c \delta^{-1-2d/\alpha} \Big(\frac{|z|}{t^{1/\alpha}}\vee 1\Big)^{-\alpha-\gamma}\Big(\frac{|z-x|}{t^{1/\alpha}}\vee 1\Big)^{-\alpha-\gamma},
	\end{equation}
	and since $\int (\tfrac{|z|}{t^{1/\alpha}}\vee 1)^{-\alpha-\gamma} (\tfrac{|z-x|}{t^{1/\alpha}}\vee 1)^{-\alpha-\gamma} \, dz< \infty$, this yields	
	$$
	  \partial_s \phi(s) = \int_{\rd} \left( p_{t-s}^{w_1}(z-x) L^{w_2} p_s^{w_2}(z)-
   p_s^{w_2}(z)L^{w_1}p_{t-s}^{w_1}(z-x)\right)\, dz, \quad s\in (0,t).  \qedhere
	$$
\end{proof}

\begin{lemma}\label{hol10}
For each $\beta\in \mathbb{N}_0^d$ and $\theta\in(0,\eta\wedge(\alpha+\gamma-d))$ there is $c>0$ such that
\begin{equation}\label{hol10-eq}
\big| \partial^\beta_x p_t^{w_1}(x)-\partial^\beta_x p_t^{w_2}(x)\big|\leq c(|w_1-w_2|^\eta\wedge 1) t^{-|\beta|/\alpha} G_t^{(\alpha+\gamma-\theta)}(x),
\end{equation}
for all $x,w_1,w_2\in \rd$, $t >0 $.
\end{lemma}
\begin{proof} Let us prove first the statement with $\beta=0$.
Since for $a,b>0$ we have $|e^{-a}-e^{-b}|\leq |a-b| e^{-(a\wedge b)}$, by H\"older continuity of $q(z,\xi)$ and  $ q(z,\xi)\approx |\xi|^\alpha$ we get
 \begin{align*}
 \big|p_t^{w_1}(x)-p_t^{w_2}(x)\big|&=(2\pi)^{-d}  \Big| \int_\rd e^{-i\xi x} \Big( e^{-t q(w_1,\xi)}-e^{-t q(w_2,\xi)}\Big) d\xi \Big| \\
 &\leq  c_1   (|w_1-w_2|^\eta\wedge 1)\Big|\int_\rd t |\xi|^\alpha e^{- c t |\xi|^\alpha} d\xi \Big|\\
 &\leq c_2 (|w_1-w_2|^\eta\wedge 1)  t^{-d/\alpha},  \quad t>0, \, w_1,w_2,x\in \rd.
 \end{align*}
Since for $|x|\leq t^{1/\alpha}$ we have $G_t^{(\alpha+\gamma-\theta)}(x)= t^{-d/\alpha}$, we get \eqref{hol10-eq} for such $t$ and $x$.

Suppose now that $|x|\geq t^{1/\alpha}$. We note that $p^z_t \in C^2_b(\Rd)\cap L^1(\Rd)$ for every $t>0$ and $z\in\rd$. Using Lemma \ref{phi_properties}, \eqref{pse2} (which yields integrability of $\partial_s \phi(s)$ on $[\delta,t-\delta]$ for every $\delta\in(0,t/2)$),
Lemma \ref{Lw_symmetry} and the  symmetry of $p_t^{w}(x)$ in $x$ we get
\begin{align*}
   p_t^{w_2}(x)&-p_t^{w_1}(x) = \int_0^t \partial_s \int_\rd p_{t-s}^{w_1}(x-z)p_s^{w_2}(z)\, dzds \\
   & = \int_0^t \int_\rd \big[p_{t-s}^{w_1}(z-x) L^{w_2} p_s^{w_2}(z) - p_s^{w_2}(z)L^{w_1}p_{t-s}^{w_1}(z-x)\big] \, dzds \\
   & = \int_0^t \int_\rd p_s^{w_2}(z) \big[L^{w_2} -L^{w_1}\big] p_{t-s}^{w_1}(z-x) \, dzds \\
   & = \int_0^t \int_\rd \big(p_s^{w_2}(z)-p_s^{w_2}(x) \big)\big[ L^{w_2} -L^{w_1}\big] p_{t-s}^{w_1}(z-x)\, dzds \\
   & \quad + \int_0^t \int_\rd p_s^{w_2}(x)\big[ L^{w_2} -L^{w_1}\big] p_{t-s}^{w_1}(z-x)\, dzds = I_1+I_2.
\end{align*}
From Lemma \ref{l:StGenerEst} and Corollary \ref{Lanihil} we have
$$
\int_\rd \big[ L^{w_2} -L^{w_1}\big] p_{t-s}^{w_1}(z-x) dz=0,
$$
hence, $I_2=0$. We next observe that for all $w,x,y\in \rd$, $t>0$,
\begin{equation}\label{lip-p}
\big|p_t^{w}(x)-p_t^{w}(y)\big|\leq c \left(\frac{|x-y|}{t^{1/\alpha}}\wedge 1\right)\left(G_t^{(\alpha+\gamma)}(x)+ G_t^{(\alpha+\gamma)}(y)\right),
\end{equation}
which follows from the Taylor expansion of $p_t^w(x)$. Indeed, if $|x-y|\geq t^{1/\alpha}$, then \eqref{lip-p} is straightforward: we just estimate the difference of functions by their sum and use Lemma~\ref{th-ptx2}.  If $|x-y|\leq t^{1/\alpha}$, then using the Taylor expansion and Lemma~\ref{th-ptx2} with $|\beta|=1$ we get
 \begin{eqnarray*}
   |p_t^{w}(x)-p_t^{w}(y)|
	 & \leq & |x-y|\cdot  \sup_{\zeta\in [0,1]} | \nabla_x p_t^w (x+\zeta(y-x))| \\
	 & \leq & c_1 |x-y| t^{-1/\alpha} \sup_{\zeta\in [0,1]} G_t^{(\alpha+\gamma)}(x+\zeta(y-x)) \\
	 & \leq & c_2 \frac{|x-y|}{t^{1/\alpha}} \left(G_t^{(\alpha+\gamma)}(x)+G_t^{(\alpha+\gamma)}(y)\right),
 \end{eqnarray*}
since for $|x|\leq 2t^{1/\alpha}$ and every $\zeta\in (0,1)$ we have
$$
G_t^{(\alpha+\gamma)}(x+\zeta(y-x))=t^{-d/\alpha}\Big(\frac{|x+\zeta(y-x)|}{t^{1/\alpha}} \vee 1\Big)^{-\alpha-\gamma} \leq t^{-d/\alpha} \leq 2^{\alpha+\gamma} G_t^{(\alpha+\gamma)}(x),
$$
and for $|x|>2t^{1/\alpha}$ we have $|x+\zeta(y-x)| \geq |x|-|y-x| \geq |x|/2$, which yields
\begin{align*}
G_t^{(\alpha+\gamma)}(x+\zeta(y-x)) &= t^{1-(d-\gamma)/\alpha}|x+\zeta(y-x)|^{-\alpha-\gamma}
\\
 &\leq 2^{\alpha+\gamma} t^{1-(d-\gamma)/\alpha} |x|^{-\alpha-\gamma} = 2^{\alpha+\gamma} G_t^{(\alpha+\gamma)}(x).
 \end{align*}
Further, using \eqref{lip-p},
\textbf{A2}, Lemma~\ref{th-ptx2} and Lemma~\ref{l:StGenerEst} with $\zeta=d$ we get
\begin{align*}
  |I_1|
  & \leq  c_1(|w_1-w_2|^\eta\wedge 1) \int_0^t \int_\rd\Big(\frac{|x-z|}{s^{1/\alpha}}\wedge 1\Big)\Big(G_s^{(\alpha+\gamma)}(x)+ G_s^{(\alpha+\gamma)}(z)\Big)\\
  &     \quad   \cdot (t-s)^{-1} G_{t-s}^{(\alpha+\gamma)}(x-z)\, dzds\\
  & \leq c_2(|w_1-w_2|^\eta\wedge 1)\int_0^t s^{-\theta/\alpha} (t-s)^{-1+\theta/\alpha}
	       \\
  &\quad \cdot \int_\rd G_{t-s}^{(\alpha+\gamma-\theta)}(x-z)\Big(G_s^{(\alpha+\gamma)}(x)+ G_s^{(\alpha+\gamma)}(z)\Big) dzds,
\end{align*}
where in the second inequality above we use the fact that
$$
  \Big(\frac{|x-z|}{s^{1/\alpha}}\wedge 1\Big) \leq \Big(\frac{t-s}{s}\Big)^{\theta/\alpha} \Big(\frac{|x-z|}{(t-s)^{1/\alpha}}\vee 1\Big)^{\theta},
	\quad x,z\in\rd,\ t>s>0,\ \theta\in (0,1].
$$
By Lemma \ref{subconvG} we obtain
$$
  |I_1| \leq c_3(|w_1-w_2|^\eta\wedge 1) \int_0^t s^{-\theta/\alpha} (t-s)^{-1+\theta/\alpha}
	\Big( G_s^{(\alpha+\gamma)}(x)+  G_t^{(\alpha+\gamma-\theta)}(x)\Big)\, ds.
$$
Note that for  $|x|^\alpha\geq t\geq s$  we have
$G_s^{(\alpha+\gamma)}(x)= s^{(\alpha+\gamma-d)/\alpha}/|x|^{\alpha+\gamma}$. Therefore,
\begin{align*}
  \int_0^t s^{-\theta/\alpha} &(t-s)^{-1+\theta/\alpha}
\Big( G_s^{(\alpha+\gamma)}(x)+  G_t^{(\alpha+\gamma-\theta)}(x)\Big)\, ds \\
 & =  B(\tfrac{2\alpha+\gamma-d-\theta}{\alpha},\tfrac{\theta}{\alpha})
	  \frac{t^{(\alpha+\gamma-d)/\alpha}}{|x|^{\alpha+\gamma}}+
		\frac{\pi}{\sin(\pi\theta/\alpha)}G_t^{(\alpha+\gamma-\theta)}(x) \\
 & \leq  c \Big(G_t^{(\alpha+\gamma)}(x) + G_t^{(\alpha+\gamma-\theta)}(x)\Big) \leq\, 2c\, G_t^{(\alpha+\gamma-\theta)}(x).
\end{align*}
Thus we have \eqref{hol10-eq} also for $|x|\geq t^{1/\alpha}$.

To prove the statement for $|\beta|\geq 1$, denote by  $h_t(x)$  the convolution semigroup corresponding to the L\'evy measure $(2M_0)^{-1}\nu_0(du)$ and denote by $\tilde{h}_t^{z} (x)$ the convolution semigroup with the L\'evy measure
$\nu^\#(z,du)=\nu(z,du)-(2 M_0)^{-1}\nu_0(du)$. We note that $p_t^{z}(x) = h_t * \tilde{h}_t^{z} (x)$ and $\nu^\#$ satisfies \textbf{A2} with constant $2M_0$ instead of $M_0$ and therefore \eqref{hol10-eq}
holds also for $\tilde{h}_t^{z} (x)$ with $\beta=0$ (and perhaps different constant $c$).
Hence for $\beta\in \N_0^d$, using Lemma \ref{th-ptx2} for $h_t$ and Lemma~\ref{subconvG} we get 
\begin{align*}
  \big| \partial^\beta_x  p_t^{w_1}(x)
	& - \partial^\beta_x  p_t^{w_2}(x)\big|
	  = \big| \partial^\beta_x \int_\rd h_t(x-y)\big( \tilde{h}_t^{w_1}(y)-\tilde{h}_t^{w_2}(y)\big)\, dy\big| \\
  & = \big| \int_\rd \partial^\beta_x  h_t(x-y)\big( \tilde{h}_t^{w_1}(y)-\tilde{h}_t^{w_2}(y)\big)\, dy\big| \\
  & \leq c_1 (|w_1-w_2|^\eta \wedge 1)   t^{-|\beta|/\alpha}  \int_\rd G_t^{(\alpha+\gamma)}(x-y)G_t^{(\alpha+\gamma-\theta)}(y)\, dy \\
  & \leq c_2 (|w_1-w_2|^\eta \wedge 1) t^{-|\beta|/\alpha} G_t^{(\alpha+\gamma-\theta)}(x).
\end{align*}
This finishes the proof.
\end{proof}
Here is a similar continuity property.
\begin{lemma}\label{pty_cont}
  For all $t>0$ and $y\in\Rd$ we have
	\begin{equation}\label{e_pty_cont}
		\lim_{z\to y} \sup_{x\in\Rd} |p_t^z (z-x) - p_t^y(y-x)| = 0.
	\end{equation}
\end{lemma}
\begin{proof}
  From Lemma \ref{hol10} and \eqref{lip-p},
we get
  \begin{eqnarray*}
	  |p_t^z (z-x) - p_t^y(y-x)|
		& \leq &  |p_t^z (z-x) - p_t^y(z-x)| + |p_t^y (z-x) - p_t^y(y-x)| \\
		& \leq & c_1 (|z-y|^\eta\wedge 1) G_t^{(\alpha+\gamma-\theta)}(z-x) \\
		&       +&  c_2 \left(\frac{|z-y|}{t^{1/\alpha}}\wedge 1\right)
		         \left(G_t^{(\alpha+\gamma)}(z-x)+ G_t^{(\alpha+\gamma)}(y-x)\right) \\
		& \leq & c_1 t^{-d/\alpha} (|z-y|^\eta\wedge 1) + 2 c_2 \left(\frac{|z-y|}{t^{1/\alpha}}\wedge 1\right) t^{-d/\alpha},
	\end{eqnarray*}
	and \eqref{e_pty_cont} follows.
\end{proof}

Lemma~\ref{hol10} also yields Lemma~\ref{delta} and \ref{derp} below. 
We have the following  result on strong continuity of $p_t^0(x,y)= p_t^y(y-x)$.
\begin{lemma}\label{delta}
For every $f\in C_0(\rd)$,
$
  \lim_{t\to 0} \sup_x \Big| \int_\rd p^y_t(y-x)f(y)\, dy - f(x)\Big|=0.
$
\end{lemma}

\begin{proof} We have
\begin{align*}
 & \Big|\int_\rd p_t^y(y-x)f(y)\,
	dy - f(x)\Big|  \\
	& \leq \Big|\int_\rd p_t^x(y-x)f(y)\, dy- f(x)\Big| + \int_\rd |p_t^y(y-x)-p_t^x(y-x)||f(y)|\, dy \\
	& \leq   \int_\rd |f(y)- f(x) | p_t^x(y-x) \, dy + \int_\rd |p_t^y(y-x)-p_t^x(y-x)||f(y)|\, dy \\
  &  =    I_1(t)+ I_2(t).
\end{align*}
Let $\delta>0$. Using Lemma \ref{th-ptx2}
for every $t\in (0,\delta^{\alpha})$ we obtain
\begin{eqnarray*}
  I_1(t)
	& \leq  & c_1 \int_\rd |f(y) - f(x)| G_t^{(\alpha+\gamma)} (y-x)\, dy =    c_1 \int_{|y-x|\leq t^{1/\alpha}} |f(y)-f(x)| t^{-d/\alpha}\, dy\\
	         & &+ c_1 \int_{t^{1/\alpha} < |y-x|\leq \delta} |f(y)-f(x)| t^{1-(d-\gamma)/\alpha}|y-x|^{-\alpha-\gamma}\, dy \\
	&       & + c_1 \int_{|y-x| > \delta} |f(y)-f(x)| t^{1-(d-\gamma)/\alpha}|y-x|^{-\alpha-\gamma}\, dy \\
	& \leq  & c_2 \sup_{|y-x|\leq \delta} |f(y)-f(x)| + c_3 \|f\|_\infty t^{1-(d-\gamma)/\alpha}\delta^{-\alpha-\gamma+d}.
\end{eqnarray*}
  Taking $\delta>0$ such that $|f(y)-f(x)|\leq \varepsilon/(2c_2)$ for $|y-x|\leq\delta$, and
	$t_0$ such that
$$
c_3 \|f\|_\infty t^{1-(d-\gamma)/\alpha}\delta^{-\alpha-\gamma+d} \leq \varepsilon/2 \quad \text{ for $t\in (0,t_0)$},
$$
	we get $\sup_{x\in\rd} I_1(t,x) \leq \varepsilon$, hence $\sup_{x\in\rd} I_1(t,x)\to 0$, as $t\to 0$.
To estimate $I_2(t,x)$ we take $\epsilon\in (\frac{d}{d+\eta},1)$ and by Lemma \ref{hol10} for every $\theta\in(0,\eta\wedge\alpha\wedge(\alpha+\gamma-d))$ we get
 \begin{align}\label{cpp}
   I_2(t)
	& = \left(\int_{|y-x|\leq t^{\epsilon/\alpha} }
	    + \int_{|y-x|> t^{\epsilon/\alpha}  }\right)
        |p_t^y(y-x)-p_t^x(y-x)| |f(y)|\, dy \\ \nonumber
  & \leq c_1 \|f\|_\infty t^{\epsilon(d+\eta)/\alpha-d/\alpha} + c_1 \|f\|_\infty \int_{|y-x|> t^{\epsilon/\alpha}}
	        G_t^{(\alpha+\gamma-\theta)}(y-x)\, dy \\ \nonumber
	       & = c_1 \|f\|_\infty t^{\epsilon(d+\eta)/\alpha-d/\alpha} + c_1 \|f\|_\infty \int_{|z|>t^{(\epsilon-1)/\alpha}}
	       (|z|\vee 1)^{-\gamma-\alpha+\theta}\, dz.
 \end{align}
 By our choice of $\epsilon$, both terms tend to 0 as $t\to 0$.
\end{proof}
We now point out the impact of cancellations, cf. Lemma~\ref{l:dertxpest}.
\begin{lemma}\label{derp} For every $\theta\in(0,\eta\wedge\frac{\alpha+\gamma-d}{2})$ we have
\begin{equation}\label{derpeq}
  \Big| \int_\rd \partial_t p_t^y(y-x)dy\Big|\leq c t^{-1+\theta/\alpha},\quad t>0.
\end{equation}
\end{lemma}
\begin{proof}
Using the fact that $\partial_t p_t^z (x)= L^z p_t^z(x)$ we get
\begin{align*}
  \int_\rd \partial_t p_t^z(z-x)dz&=  \int_\rd L^z p_t^z(z-x)\, dz
   =  \int_\rd L^z \big( p_t^z(z-x)-p_t^x(z-x)\big)\, dz\\
   &+ \int_\rd \big(L^z -L^x\big) p_t^x(z-x)\, dz  \quad +\int_\rd L^x p_t^x(z-x)\, dz.
\end{align*}
Lemma \ref{Lanihil} yields $\int_\rd L^x p_t^x(z-x)\, dz=0$.  Further, by Lemma~\ref{hol10}  and \ref{l:StGenerEst} we get
\begin{align*}
  \int_\rd \big| L^z (p_t^z(z-x)- p_t^x(z-x))\big|dz
	&\leq c_1\int_\rd (|x-z|^\eta\wedge 1)t^{-1} G_t^{(\alpha+\gamma-\theta)}(z-x)\, dz \\
  &   =  c_1 t^{-1-d/\alpha} \int_\rd (|y|^\eta \wedge 1) (1\vee (t^{-1/\alpha}|y|))^{-\gamma-\alpha+\theta}\, dy \\
  & \leq c_2 t^{-1+(\eta \wedge (\alpha+\gamma-\theta-d))/\alpha} \leq c_2 t^{-1+\theta/\alpha}.
\end{align*}
Similarly, by  \textbf{A2}, Lemma~\ref{th-ptx2} and Lemma~\ref{l:StGenerEst} we get
\begin{align*}
  \int_\rd \big|\big(L^z -L^x\big) p_t^x(z-x)\big| \, dz
	& \leq c_1 t^{-1-d/\alpha} \int_\rd (|z-x|^{\eta} \wedge 1)  G_t^{(\alpha+\gamma)}(z-x) \, dz \\
	& \leq c_2 t^{-1+(\eta \wedge (\alpha+\gamma-d))/\alpha} \leq c_2 t^{-1+\theta/\alpha},
\end{align*}
which finishes the proof of the lemma.
\end{proof}

\section{Parametrix
}\label{para}

\subsection{Proof of convergence}
In this section we prove that $p_t$ given by \eqref{r} is well defined.
To this end we let
\begin{equation}\label{Psi2}
\Psi^\#_t(x,y):=\sum_{k=1}^\infty  |\Phi|_t^{\boxtimes k}(x,y),
\end{equation}
and
\begin{equation}\label{pmod}
p^\#_t(x,y)= p^0_t(x,y)+  \big(p^0\boxtimes \Psi^\# \big)_t(x,y).
\end{equation}
Our first result shows that the series \eqref{Psi2} and the function $p_t^\#(x,y)$ are finite,  and have nice estimates. Then $p_t(x,y)$ is well defined, with the same upper bounds.

\begin{proposition}\label{pr1}
The series \eqref{Psi2} is convergent, the integral $p^0\boxtimes \Psi^\#$  exists, and
 \begin{equation}\label{up10}
 p^\#_t(x,y)\leq C e^{ct} G_t^{(\alpha+\gamma)}(y-x), \quad t>0,\, x,y\in \rd.
 \end{equation}
\end{proposition}
The result depends on  auxiliary estimates of $|\Phi_t(x,y)|$ and its convolutions, which we now give. The proof of Proposition~\ref{pr1} is found at the end of the next subsection.

\begin{lemma}\label{Phi-up}  Under condition \textbf{A2} there exists $C_\Phi>0$ such that
  \begin{equation}\label{Phi10}
	  \big|\Phi_t(x,y)\big| \leq  C_\Phi t^{-1} (1\wedge |y-x|^\eta) G_t^{(\alpha+\gamma)}(y-x), \quad x,y\in\Rd,\, t>0,
  \end{equation}
	the function $\partial_t \Phi_t(x,y)$ exists for all $t>0,\, x,y\in\Rd,$ is continuous in  $t$, and
	\begin{equation}\label{Phi101}
	  \big|\partial_t \Phi_t(x,y)\big| \leq  C_\Phi t^{-2} (1\wedge |y-x|^\eta) G_t^{(\alpha+\gamma)}(y-x), \quad x,y\in\Rd,\, t>0.
  \end{equation}
\end{lemma}
\begin{proof} By the symmetry of $\nu(x,\cdot)$ for every  $ x\in \rd$  and  \textbf{A2}, $\big|\Phi_t(x,y)\big|$ equals
  \begin{align*}
	           &    \left| \int \left(p_t^0(x+u,y)- p_t^0 (x,y)-\scalp{u}{\nabla_x p_t^0(x,y)} \indyk{\{|u|\leq t^{1/\alpha}\}} \right) \,
	          \left(\nu(x,du) - \nu(y,du)\right) \right| \\
&\leq  C \left( |y-x|^{\eta}\wedge 1\right)\modgener   p_t^0(x,y).
  \end{align*}
  By Lemma~\ref{th-ptx2} and Lemma~\ref{l:StGenerEst} with $\zeta=d$ and $\kappa=0$ we get \eqref{Phi10}. The estimate \eqref{Phi101}
	follows from Lemma~\ref{l:StGenerEst} and the fact that  $\partial_t \Phi_t(x,y)$ equals
	\begin{equation*}\label{diffofPhi}
\begin{split}
  &  \int  \partial_t  \left( p_t^0(x+u,y)-  p_t^0 (x,y)-\scalp{u}{\nabla_x  p_t^0(x,y)} \indyk{|u|\leq 1} \right)
	          \left(\nu(x,du) - \nu(y,du)\right).
\end{split}
  \end{equation*}
  We can change the order of differentiation and integration since
	by Lemma~\ref{l:dertxpest} for every $t>0$ and $\epsilon\in (-t/2,t/2)$ we have
\begin{align*}
    \big|\partial_t &p_{t+\epsilon}^0(x+u,y)- \partial_t p_{t+\epsilon}^0 (x,y)-\scalp{u}{\nabla_x \partial_t p_{t+\epsilon}^0(x,y)} \indyk{\{|u|\leq 1\}} \big| \\
	 & \leq c_1 (t+\epsilon)^{-1-(d+2)/\alpha} |u|^2 \, \indyk{\{|u|\leq 1\}} + c_2 (t+\epsilon)^{-1-d/\alpha} \indyk{\{|u| > 1\}} \\
	 & \leq c_3 t^{-1-(d+2)/\alpha} |u|^2 \, \indyk{\{|u|\leq 1\}} +  c_4 t^{-1-d/\alpha}\indyk{\{|u| > 1\}} \\
	 &   = : g_t(u) ,\quad u\in\Rd, x,y\in\Rd,
\end{align*}
and $\int_{\Rd} g_t(u)\, \nu(w,du) < \infty$ for every $w\in\Rd$. This yields \eqref{diffofPhi}
and the continuity  of $t\mapsto \partial_t \Phi_t(x,y)$.
\end{proof}

To estimate $\Phi^{\boxtimes k}$ we will use the following  \emph{sub-convolution} property.

\begin{definition}  A non-negative kernel $H_{t}(x), t>0, x\in \rd,$ has the sub-convolution property if
 there is a constant $C_{H}>0$ such that
\begin{equation}\label{H0}
(H_{t-s} * H_s)(x)\leq C_{H} H_{t}(x), \quad 0<s< t , \quad x\in \rd.
\end{equation}
\end{definition}
It follows from Lemma~\ref{subconvG} that $G_t^{(\beta)}(x)$  has the sub-convolution property. On the other hand,  the kernel  $t^{-1}(1\wedge |x|^\eta) G^{(\alpha+\gamma)}_t (x)$  from Lemma~\ref{Phi-up} does not have it; take for instance $x=0$ in \eqref{H0} or see \cite{2014arXiv1412.8732K} in the case when $d=\gamma$. To circumvent this problem,  for $\zeta>0$ and $\kappa\in (d-\alpha,d]$ we  define
\begin{equation}\label{H10}
  H_t^{(\kappa,\zeta)}(x)= \left(t^{-\zeta /\alpha}\wedge \Big(\frac{|x|}{t^{1/\alpha}}\vee 1 \Big)^\zeta \right)G_t^{(\alpha+\kappa)}(x).
\end{equation}
\begin{proposition}\label{sub-conv} Assume that
\begin{equation}\label{algam}
\alpha+\kappa-d>\zeta.
\end{equation}
The kernels $H_t^{(\kappa,\zeta)}(x)$ satisfy the sub-convolution property with  some constant $C_H$
and there exists a positive constant $C>0$ such that
\begin{equation}\label{interg}
  \int_\rd H_t^{(\kappa,\zeta)} (x)dx \leq C, \quad t>0.
\end{equation}
\end{proposition}
\begin{proof}
The proof follows that of \cite[Proposition 3.3]{2014arXiv1412.8732K}.
We have
\begin{equation}\label{H20}
  H_t^{(\kappa,\zeta)}(x) \leq (t^{-\zeta/\alpha}\wedge 1) G_t^{(\kappa+ \alpha-\zeta)}(x),\quad x\in\Rd, t>0,
\end{equation}
and
\begin{equation}\label{H201}
  H_t^{(\kappa,\zeta)}(x)  = (t^{-\zeta/\alpha}\wedge 1)\, G_t^{(\kappa+ \alpha-\zeta)}(x),
	\quad |x|\leq 1 \vee t^{1/\alpha},t>0.
\end{equation}
Clearly, \eqref{H20} implies that
\begin{equation}\label{H30}
\int_\rd H_t^{(\kappa,\zeta)}(dx)\leq C.
\end{equation}
We notice that
$$
  ( (t-s)^{-\zeta/\alpha} \wedge 1 ) (s^{-\zeta/\alpha}\wedge 1) \leq 2^{\zeta/\alpha} (t^{-\zeta/\alpha} \wedge 1), \quad 0<s<t.
$$
By this, Lemma \ref{subconvG}, \eqref{H20} and \eqref{H201},
$$
  \Big( H_{t-s}^{(\kappa, \zeta)} * H_s^{(\kappa,\zeta)}\Big)(x)\leq
	C  H_t^{(\kappa,\zeta)}(x), \quad |x|\leq 1 \vee t^{1/\alpha}, t>0.
$$
To complete the proof we assume that $|x|\geq 1 \vee t^{1/\alpha}$. We have
 $$
   \Big( H_{t-s}^{(\kappa,\zeta)} * H_s^{(\kappa,\zeta)}\Big)(x)
	 \leq  \left(\int_{|z|\geq |x|/2} +
	       \int_{|x-z|\geq |x|/2}\right)H_{t-s}^{(\kappa,\zeta)}(z) H_s^{(\kappa,\zeta)}(x-z)dz.
$$
By the structure of $H_t^{(\kappa,\zeta)}(x)$, for $|z|\geq |x|/2$ we obtain $H_{t-s}^{(\kappa,\zeta)}(z)\leq c H_{t-s}^{(\kappa,\zeta)}(x)$.
We have $H_t^{(\kappa,\zeta)}(x) = t^{1-(\zeta+d-\kappa)/\alpha} |x|^{-\kappa-\alpha}$, and by
\eqref{algam},
\begin{align*}
  H_{t-s}^{(\kappa,\zeta)}(x)
	&  = (t-s)^{1-(\zeta+d-\kappa)/\alpha} |x|^{-\kappa-\alpha}\\
	 & \leq t^{-1+(\zeta+d-\kappa)/\alpha} |x|^{-\kappa-\alpha}
  = H_t^{(\kappa,\zeta)}(x).
\end{align*}
Using \eqref{H30}
we get
\begin{align*}
\int_{|z|\geq |x|/2}H_{t-s}^{(\kappa,\zeta)}(z) &H_s^{(\kappa,\zeta)}(x-z)\,dz\leq c  H_t^{(\kappa,\zeta)}(x) \int_{|z|\geq |x|/2}H_t^{(\kappa,\zeta)}(x-z)\,dz
\\&\leq c H_t^{(\kappa,\zeta)}(x) \int_{\rd}H_t^{(\kappa,\zeta)}(z)\,dz
\leq C H_t^{(\kappa,\zeta)}(x).
\end{align*}
Similarly,
$$
\int_{|x-z|\geq |x|/2} H_{t-s}^{(\kappa,\zeta)}(z) H_s^{(\kappa,\zeta)}(x-z)dz\leq C H_t^{(\kappa,\zeta)}(x).\qedhere
$$
\end{proof}

 Let us rewrite the upper estimate in \eqref{Phi10}.  Since
  $$
    1\wedge  |x|^\theta= t^{\theta/\alpha}\left(  t^{-\theta/\alpha} \wedge\left(\frac{|x|}{t^{1/\alpha}}\right)^\theta\right)
	  \leq t^{\theta/\alpha}\left(t^{-\theta/\alpha}\wedge\left( \left( \frac{|x|}{t^{1/\alpha}}\right)^\theta \vee 1 \right)\right),
  $$
   we get
   \begin{equation}\label{GH}
 t^{-\theta/\alpha}(1\wedge |x|^\theta)  G_t^{(\alpha+\gamma)}(x)\leq H_t^{(\gamma, \theta)}(x),
 \end{equation}
 which implies for $\theta\leq \eta$
 \begin{equation}\label{Phi1}
 \big| \Phi_t(x,y)\big| \leq C_\Phi t^{-1+\theta/\alpha} H_t^{(\gamma,\theta)}(y-x), \quad x,y\in \rd, t>0.
 \end{equation}
Using the sub-convolution property of $H_t^{(\gamma,\theta)}(x)$,  we can estimate $\Phi_t^{\boxtimes k}(x,y)$. Let
\begin{equation}\label{thet}
0<\theta< \alpha \wedge \eta\wedge (\alpha+\gamma-d).
\end{equation}

\begin{lemma}\label{Phi-l20} For every $k\geq 2$ and $\theta$ satisfying \eqref{thet} we have
  \begin{equation}\label{Phi20}
    \left|\Phi\right|^{\boxtimes k}_t(x,y)\leq \frac{C_1  C_2^k}{\Gamma(k \theta/\alpha)}t^{-1+k\theta/\alpha}  H_t^{(\gamma,\theta)}(y-x), \quad x,y\in \rd, \quad t>0.
  \end{equation}
\end{lemma}
\begin{proof}   Let $C_1=C_H^{-1}$ ,  $C_2= C_\Phi C_H\Gamma(\theta/\alpha)$, where $C_\Phi$ is from \eqref{Phi1}  and $C_H$ is from {\rm Proposition \ref{sub-conv}}. We will use induction. For $k=1$ we already have \eqref{Phi1}. Suppose that \eqref{Phi20} holds for $k$.
By the sub-convolution property of $H_t^{(\gamma,\theta)}$,
  \begin{align*}
    \big|\Phi^{\boxtimes (k+1)}_t&(x,y)\big|\\
   & \leq  \frac{C_1 C_\Phi C_2^k }{\Gamma(k\theta/\alpha)}\int_0^t (t-s)^{-1+k \theta/\alpha} s^{-1+\theta/\alpha}
		   \int_{\Rd} H_{t-s}^{(\gamma,\theta)}(x-z)H_{s}^{(\gamma,\theta)}(z-y)\, dzds \\
    &  \leq \frac{C_1 C_\Phi C_H C_2^k }{\Gamma(k\theta/\alpha)} H_t^{(\gamma,\theta)}(y-x)\int_0^t
      (t-s)^{-1+k \theta/\alpha} s^{-1+\theta/\alpha}\, ds  \\
    &  = \frac{C_1 C_\Phi C_H C_2^k }{\Gamma(k\theta/\alpha)}  t^{-1+(k+1)\theta/\alpha}  H_t^{(\gamma,\theta)}(y-x)
	    \frac{\Gamma(k\theta/\alpha)\Gamma(\theta/\alpha)}{\Gamma((k+1)\theta/\alpha)}  \\
    &  = \frac{C_1 C_2^{k+1}  }{\Gamma((k+1)\theta/\alpha)} t^{-1+(k+1)\theta/\alpha} H_t^{(\gamma,\theta)}(y-x).\qedhere
  \end{align*}
\end{proof}

\begin{corollary}\label{Phik_t_cont}
   For all $x,y\in\Rd$ and $k=1,2,...$, $t\to \Phi_t^{\boxtimes k}(x,y)$ is continuous.
\end{corollary}
\begin{proof}
   For every $h\in (0,t/2)$ we have
  \begin{align*}
	   & \left| \Phi_{t+h}^{\boxtimes (k+1)}(x,y) -  \Phi_{t}^{\boxtimes (k+1)}(x,y) \right| \\
		& \leq  \int_0^{t-h} \int_{\Rd} \left| \Phi_{t+h-s}(x,z) -  \Phi_{t-s}(x,z)\right| \Phi_{s}^{\boxtimes k}(z,y)\, dz ds \\
		&      \,\, +  \int_{t-h}^{t} \int_{\Rd} \left| \Phi_{t+h-s}(x,z) -  \Phi_{t-s}(x,z)\right| \Phi_{s}^{\boxtimes k}(z,y)\, dz ds \\
		&      \,\,  + \int_{t}^{t+h} \int_{\Rd} \Phi_{t+h-s}(x,z) \Phi_{s}^{\boxtimes k}(z,y)\, dz ds = I_1(h) + I_2(h) + I_3(h) .
	\end{align*}
	Using Lemma~\ref{Phi-up}, Lemma~\ref{Phi-l20} and \eqref{GH} we obtain
	\begin{eqnarray*}
	  I_1(h)
		& \leq & c_1 h \int_0^{t-h} \int_{\Rd} (t-s)^{-2} (1\wedge |z-x|^{\theta}) G_{t-s}^{(\alpha+\gamma)} (z-x) \Phi_{s}^{\boxtimes k}(z,y)\, dz ds \\
		& \leq & c_2 h \int_0^{t-h} \int_{\Rd} (t-s)^{-2+\theta/\alpha} s^{-1+k\theta/\alpha} H_{t-s}^{(\gamma,\theta)} (z-x)   H_s^{(\gamma,\theta)}(y-z)\, dz ds \\
		& \leq & c_3 h H_{t}^{(\gamma,\theta)} (y-x) \int_0^{t-h} (t-s)^{-2+\theta/\alpha} s^{-1+k\theta/\alpha}\, ds \\
		& \leq & c_4 H_{t}^{(\gamma,\theta)} (y-x) t^{-1+k\theta/\alpha} h^{\theta/\alpha},
	\end{eqnarray*}
	and so $\lim_{h\to 0^+} I_1(h) = 0$. Furthermore,
	\begin{eqnarray*}
	  I_2(h)
		& \leq &  c_1 \int_{t-h}^{t} \int_{\Rd} (1\wedge |z-x|^{\theta}) (t+h-s)^{-1}G_{t+h-s}^{(\alpha+\gamma)}(z-x) \Phi_{s}^{\boxtimes k}(z,y)\, dz ds\\
		&      & + \,\, c_2 \int_{t-h}^{t} \int_{\Rd} (1\wedge |z-x|^{\theta}) (t-s)^{-1}G_{t-s}^{(\alpha+\gamma)}(z-x)
		           \Phi_{s}^{\boxtimes k}(z,y)\, dz ds \\
	  & \leq & c_3 H_{t+h}^{(\gamma,\theta)}(y-x) \int_{t-h}^{t} s^{-1+k\theta/\alpha}  (t+h-s)^{-1+\theta/\alpha}\, ds \\
		&      & +  \,\, c_4 H_{t}^{(\gamma,\theta)}(y-x) \int_{t-h}^{t} s^{-1+k\theta/\alpha} (t-s)^{-1+\theta/\alpha} \, ds \\
		& \leq & c_5 \left(H_{t+h}^{(\gamma,\theta)}(y-x) + H_{t}^{(\gamma,\theta)}(y-x) \right)t^{-1+k\theta/\alpha} h^{\theta/\alpha}.
	\end{eqnarray*}
	Similarly we obtain
	\begin{eqnarray*}
	  I_3(h)
		& \leq & c_1 \int_{t}^{t+h} \int_{\Rd} (1\wedge |z-x|^{\theta}) (t+h-s)^{-1}G_{t+h-s}^{(\alpha+\gamma)}(z-x) \Phi_{s}^{\boxtimes k}(z,y)\, dz ds \\
	  & \leq & c_2 \int_{t}^{t+h} \int_{\Rd} s^{-1+k\theta/\alpha}  (t+h-s)^{-1+\theta/\alpha}H_{t+h-s}^{(\gamma,\theta)}(z-x)  H_s^{(\gamma,\theta)}(y-z)\, dz ds \\
		& \leq & c_3 H_{t+h}^{(\gamma,\theta)}(y-x) t^{-1+k\theta/\alpha} h^{\theta/\alpha},
	\end{eqnarray*}
	so $\lim_{h\to 0^+} I_2(h) = \lim_{h\to 0^+} I_3(h) = 0$ (and analogously for negative $h$).
\end{proof}

\begin{proof}[Proof of Proposition~\ref{pr1}]

 By Lemma~\ref{Phi-l20}, the series  $\Psi^\#_t(x,y)= \sum_{m=1}^\infty |\Phi|_t^{(\boxtimes k)}(x,y)$  converges uniformly on  compact subsets of $(0,\infty)\times \rd\times\rd$. 
By the sub-convolution property of $H_t^{(\gamma,\theta)}(x)$,  \eqref{Phi1}, \eqref{Phi20},  and the estimate
\begin{equation}\label{eq:Get}
   \sum_{k=0}^\infty \frac{C_2^k t^{\zeta k}}{\Gamma((k+1)\zeta)}\leq c_1 e^{c_2 t},\quad \zeta>0,t>0,
\end{equation}
for which see, e.g., \cite{MR2929183}, we get
\begin{equation}\label{Psi1}
  |\Psi_t(x,y)|\le \Psi_t^\#(x,y)  \leq   c_3 t^{-1+\theta/\alpha}e^{c_2 t}H_t^{(\gamma,\theta)}(y-x), \quad x,y\in \Rd,\quad t>0.
\end{equation}
For every $t>0$ we have
\begin{equation}\label{GH22}
  G_s^{(\alpha+\gamma)}(x) \leq \frac{1}{t^{-\theta/\alpha}\wedge 1} H_s^{(\gamma,\theta)}(x),\quad
	x\in\Rd, s\in (0,t].
\end{equation}
Then, for $x,y\in \rd$, $t>0$,
\begin{eqnarray}
  &&\left|\big(p^0\boxtimes \Psi\big)_t(x,y)\right|\le   \big(p^0\boxtimes \Psi^\#\big)_t(x,y) \label{p0-psi}\\
  &&=    \int_0^t \int_{\Rd} p^0_{t-s}(x,z) \Psi_s^{\#}(z,y)\, dz ds  \nonumber \\ \nonumber
  & &\leq  c_4 \int_0^t s^{-1+\theta/\alpha}e^{c_2 s} \int_{\Rd} G_{t-s}^{(\alpha+\gamma)}(z-x) H_s^{(\gamma,\theta)}(y-z)\, dz ds \nonumber\\
	&& \leq  c_4 \int_0^t  \frac{s^{-1+\theta/\alpha}e^{c_2 s}}{t^{-\theta/\alpha}\wedge 1}\int_{\Rd} H_{t-s}^{(\gamma,\theta)}(z-x) H_s^{(\gamma,\theta)}(y-z)\, dz ds \nonumber \\
	&& \leq  c_5 t^{\theta/\alpha}e^{c_2 t} H_t^{(\gamma,\theta)}(y-x)  \label{e:small}\leq  c_5 e^{c_2 t} G_t^{(\alpha+\gamma)}(y-x), 
\end{eqnarray}
which follows from \eqref{H20}. This proves \eqref{up10}.
\end{proof}
From \eqref{e:small} we see in particular that $p_t(x,y)$ is well defined.
\begin{lemma} The following perturbation formula holds for all $t>0$, $x,y\in \Rd$,
	  \begin{equation}\label{eq:pfo}
		  p_t(x,y) = p^0_t(x,y) + \int_0^t \int_{\Rd} p_s(x,z)\Phi_{t-s}(z,y)\, dz ds.
		\end{equation}
\end{lemma}
\begin{proof}
The identity follows from \eqref{r}, \eqref{Psi} and Proposition~\ref{pr1}.
\end{proof}
We note a qualitative difference between \eqref{eq:pfo} and \eqref{r}.
\subsection{Regularity of
$\Psi_s(x,y)$ and $p_t(x,y)$}

The statement of Proposition~\ref{pr1} implies the existence of the function $p_t(x,y)$.  In this section we establish the H\"older continuity in $x$ of the function $\Psi$ and a few auxiliary results for the proof Theorem~\ref{t-exist}.

\begin{lemma}\label{phol} For all $\epsilon\in (0,\theta)$, where $\theta$ satisfies \eqref{thet} and $T>0$  there exists $C=C(T)>0$ such that
for all $t\in (0,T]$, $x_1,x_2,y\in \rd$,
\begin{equation*}\label{psol-eq1}
  \big| \Psi_t(x_1,y)-\Psi_t(x_2,y)\big| \leq  C
  \big(|x_1-x_2|^{\theta-\epsilon}\wedge 1\big) t^{-1+\epsilon/\alpha} \Big(
  H_t^{(\gamma,\theta)}(y-x_1)+H_t^{(\gamma,\theta)}(y-x_2)\Big).
\end{equation*}
\end{lemma}
\begin{proof}
 We begin by proving that for $t\in (0,1]$, $x_1,x_2,y\in \rd$,
\begin{equation*}\label{phol-eq1}
\big| \Phi_t(x_1,y)-\Phi_t(x_2,y)\big| \leq  C \big(|x_1-x_2|^{\theta-\epsilon}\wedge 1\big) t^{-1+\epsilon/\alpha} \Big(
H_t^{(\gamma,\theta)}(y-x_1)+H_t^{(\gamma,\theta)}(y-x_2)\Big).
\end{equation*}
For $|x_1-x_2|\geq 1$ the estimate simply follows from \eqref{Phi1}.
Suppose now that \newline $t^{1/\alpha}\leq |x_1-x_2|\leq 1$. Then,
\begin{align*}
\big| \Phi_t(x_1,y)-\Phi_t(x_2,y)\big| &\leq
c_1 \big(\big| \Phi_t(x_1,y)\big| + \big|\Phi_t(x_2,y)\big|\big) \\
&\leq c_2 t^{-1+\theta/\alpha} \big( H_t^{(\gamma,\theta)}(y-x_1)+H_t^{(\gamma,\theta)}(y-x_2)\big)\\
&\leq c_3|x_1-x_2|^{\theta-\epsilon} t^{-1+\epsilon/\alpha} \big( H_t^{(\gamma,\theta)}(y-x_1)+H_t^{(\gamma,\theta)}(y-x_2)\big).
\end{align*}
Let $|x_1-x_2|\leq t^{1/\alpha}\wedge 1$ and
$$
g(x,y, u)= p_t^0(x+u,y)-p_t^0(x,y)-u\cdot \nabla_x p_t^0 (x,y) \indyk{\{|u|\leq t^{1/\alpha}\}}.
$$
We have
\begin{align*}
&  \Phi_t(x_1,y)-\Phi_t(x_2,y)
	 = \int_\rd g(x_1,y, u) [\nu(x_1,du)-\nu(x_2,du)] \\
  & \quad + \int_\rd \big(g(x_1,y, u)-g(x_2,y, u) \big) [\nu(x_2,du)-\nu(y,du)]= I_1+I_2.
\end{align*}
For $I_1$ by \textbf{A2}, Lemma~\ref{th-ptx2} and Lemma~\ref{l:StGenerEst} with $\zeta=d$, $\kappa=0$, we get
\begin{align*}
 |I_1|
 & \leq c_1 (|x_1-x_2|^\theta\wedge 1)\modgener p_t^0(x_1,y)  \leq c_2 (|x_1-x_2|^\theta\wedge 1) t^{-1} G_t^{(\alpha+\gamma)}(y-x_1).
\end{align*}
To estimate $I_2$ let
$$
  f_t(x) =  p_t^y(x+x_2-x_1) - p_t^y(x).
$$
Using the Taylor expansion, Lemma~\ref{th-ptx2} and the fact that $|x_2-x_1|\leq t^{1/\alpha}$, we get
\begin{eqnarray*}
  |f_t(x)|
	&   =  & \left|(x_2-x_1) \cdot\nabla_x p_t^y(x+\zeta (x_2-x_1))\right| \\
	& \leq & c_1 |x_2-x_1| t^{-1/\alpha} G_t^{\alpha+\gamma}\left(x+\zeta (x_2-x_1)\right) \\
	& \leq & c_2 |x_2-x_1| t^{-1/\alpha} G_t^{\alpha+\gamma}(x), \quad x\in\Rd, \ t>0,
\end{eqnarray*}
where we used some $\zeta\in [0,1]$.
Similarly, if $\beta\in\N_0^d$, $|\beta|=2$, then
$$
  |\partial_x^{\beta} f_t(x)| \leq c |x_2-x_1| t^{-3/\alpha} G_t^{\alpha+\gamma}(x).
$$
By \textbf{A2} and Lemma~\ref{l:StGenerEst} (applied with $\zeta=d+1$, $\kappa=0$) we get
\begin{equation}\label{I2}
\begin{split}
  |I_2|
	& \leq M_0 (|x_2-y|^\theta \wedge 1) \modgener f_t (y-x_2) \\
  & \leq c_2 (|x_2-y|^\theta \wedge 1) |x_2-x_1|t^{-1-1/\alpha} G_t^{(\alpha+\gamma)}(y-x_2).
\end{split}
\end{equation}
Then for  $|x_1-x_2|\leq t^{1/\alpha} \wedge 1$,  $t\in (0,T]$, using the inequality
\begin{equation}\label{gt}
  G_t^{(\alpha+\gamma)}(x)\leq (T^{\theta/\alpha} \vee 1) H_t^{(\gamma,\theta)}(x),
\end{equation}
we get
$$
  |I_1| \leq c_1 (T^{\theta/\alpha} \vee 1) |x_1-x_2|^{\theta-\epsilon}  t^{-1+\epsilon/\alpha} H_t^{(\gamma,\theta)}(y-x_1).
$$
Furthermore, using \eqref{GH} we obtain
\begin{eqnarray*}
  |I_2|
	& \leq & c_1 (|x_2-y|^\theta \wedge 1) |x_2-x_1|t^{-1-1/\alpha} G_t^{(\alpha+\gamma)}(y-x_2) \\
  & \leq & c_1 |x_2-x_1|t^{-1-1/\alpha+\theta/\alpha}
	         H_t^{(\gamma,\theta)}(y-x_2) \\
	&   =  & c_1 |x_2-x_1|^{\theta-\epsilon} |x_2-x_1|^{1-\theta+\epsilon} t^{-1-1/\alpha+\theta/\alpha}
	         H_t^{(\gamma,\theta)}(y-x_2) \\
  & \leq & c_1 |x_2-x_1|^{\theta-\epsilon} t^{-1+\epsilon/\alpha}H_t^{(\gamma,\theta)}(y-x_2).
\end{eqnarray*}
We now prove
the inequality in the statement of the lemma.
For $|x_1-x_2|\geq 1$ the estimate follows from the bound \eqref{Psi1} on $\Psi_t(x,y)$, so we let  $|x_1-x_2|\leq 1$.
Since
$$
\Psi_t(x,y)= \Phi_t(x,y)+ (\Phi\boxtimes\Psi)_t(x,y),
$$
by Proposition~\ref{sub-conv} for $t\in (0,T]$ we get
\begin{align*}
\Big| \Psi_t&(x_1,y)-\Psi_t(x_2,y) \Big|
 \leq  c_1 |x_1-x_2|^{\theta-\epsilon} t^{-1+\epsilon/\alpha}
          \Big( H_t^{(\gamma,\theta)}(y-x_1)+H_t^{(\gamma,\theta)}(y-x_2) \Big) \\
&       +\,\, c_2 |x_1-x_2|^{\theta-\epsilon}\int_0^t \int_\rd (t-s)^{-1+\epsilon/\alpha}
          \Big( H_{t-s}^{(\gamma,\theta)}(z-x_1)+H_{t-s}^{(\gamma,\theta)}(z-x_2)\Big) \\
&        \cdot\, s^{-1+\theta/\alpha} H_{s}^{(\gamma,\theta)}(y-z)\, dzds \\
& \leq  c_3  |x_1-x_2|^{\theta-\epsilon} t^{-1+\epsilon/\alpha}
          \Big( H_t^{(\gamma,\theta)}(y-x_1)+H_t^{(\gamma,\theta)}(y-x_2) \Big).\qedhere
\end{align*}
\end{proof}
We can finally apply the operator $L$ to $p_t(x,y)$.
\begin{lemma}\label{lem-ltil}
  For all $y\in\Rd$ and $t>0$ we have $p_t(\cdot,y) \in D(L)$, and
\begin{equation}\label{ltil}
  L_x p_t(x,y) = L_x p_t^0(x,y)+ \int_0^t\int_\rd L_x p_{t-s}^0(x,z)\Psi_s(z,y)\, dzds.
\end{equation}
\end{lemma}
\begin{proof}
Since $p_t^0(\cdot,y)\in C^2_\infty(\rd)$, the term  $L_x p_t^0(x,y)$ is well defined. Using the representation of $L^{x,\delta} $, Lemma~\ref{th-ptx2} and Lemma~\ref{l:StGenerEst}, for every $\delta>0$
we get
\begin{equation*}\label{I0-est}
\begin{split}
  & |L^\delta_x p_t^0(x,y)|
 	  \leq  \int_{|u|>\delta} \left| p_t^0(x+u,y) - p_t^0(x,y) - \scalp{u}{\nabla_x p_t^0(x,y)}\indyk{ \{|u|\leq t^{1/\alpha}\} }\right|\, \nu(x,du)  \\
	& \leq  \modgener p_t^0(x,y) \leq c t^{-1} G_t^{(\alpha+\gamma)}(y-x).
\end{split}
\end{equation*}
Let us show that the function
\begin{equation}\label{fy}
  f_t^y(x):=\int_0^t\int_\rd  p_{t-s}^0(x,z)\Psi_s(z,y)\, dzds
\end{equation}
belongs to $D(L)$ and
$$
  L_x f_t^y(x)=\int_0^t\int_\rd L_x p_{t-s}^0(x,z)\Psi_s(z,y)\, dzds.
$$
For this end we use the definition \eqref{lxd0} of $L_x $. By \eqref{p0-psi}
for every $\delta>0$ we get
\begin{align*}
  \int_{|u|>\delta} \int_0^t \int _{\Rd}&
	\left|\left(p_{t-s}^0(x+u,z) - p_{t-s}^0(x,z)\right) \Psi_s(z,y) \right|\,
	dzds\nu(x,du)  \\
	&\leq c_1 \int_{|u|>\delta} ( G_t^{(\alpha+\gamma)}(y-x-u)+G_t^{(\alpha+\gamma)}(y-x) )\, \nu(x,du)\\
	& < c_2 t^{-d/\alpha} \nu_0(B(0,\delta)^c).
\end{align*}
By Fubini's theorem and the symmetry of $\nu$ we get
\begin{align*}
  L^{\delta} f_t^y(x)
	& = \int_0^t \int_\rd L^{\delta}_x p_{t-s}^0 (x,z) \Psi_s (z,y)\, dzds \\
  & = \int_0^t \int_\rd L^{\delta}_x p_{t-s}^0 (x,z) \big[\Psi_s (z,y)- \Psi_s(x,y)\big] \, dzds \\
  &   \quad + \int_0^t \int_\rd L^{\delta}_x p_{t-s}^0 (x,z)\Psi_s(x,y)\, dz ds \\
	& = \int_0^t \int_\rd L^{\delta}_x p_{t-s}^0 (x,z) \big[\Psi_s (z,y)- \Psi_s(x,y)\big] \, dzds \\
  &   \quad + \int_0^t \int_\rd L^{x,\delta} (p_{t-s}^z - p_{t-s}^x) (z-x) \Psi_s(x,y)\,dz ds \\
	&   \quad + \int_0^t \int_\rd L^{x,\delta} p_{t-s}^x (z-x)\Psi_s(x,y)\, dz ds
	     = I_1(\delta)+I_2(\delta) + I_3(\delta).
\end{align*}
Let us estimate the functions under the integrals $I_1(\delta)$ and $I_2(\delta)$.
 Using Lemma~\ref{phol} and \eqref{GH} for $T>0$ and $0<s<t\le T$ we get
\begin{align*}
  &\left| L^{\delta}_x p_{t-s}^0 (x,z) \big[\Psi_s (z,y)- \Psi_s(x,y)\big] \right|
	  \\
	&\leq c_1 (t-s)^{-1} G_{t-s}^{(\alpha+\gamma)}(z-x)
	   \big(|x-z|^{\theta-\epsilon}\wedge 1\big) s^{-1+\epsilon/\alpha}
  \big[ H_s^{(\gamma,\theta)}(y-z)+ H_s^{(\gamma,\theta)}(y-x)\big] \\
	& \leq  c_2 s^{-1+\epsilon/\alpha} (t-s)^{-1+(\theta-\epsilon)/\alpha}
	         H_{t-s}^{(\gamma,\theta-\epsilon)}(z-x)
	         \big[ H_s^{(\gamma,\theta)}(y-z)+ H_s^{(\gamma,\theta)}(y-x)\big]  \\
	         &=:  g_t^{(x,y)}(s,z),
\end{align*}
with $c_1,c_2>0$ depending on $T$. Using Proposition~\ref{sub-conv}  and  the inequality
$$H_{t-s}^{(\gamma,\theta-\epsilon)}(z-x) \leq (T^{\epsilon/\alpha}\vee 1) H_{t-s}^{(\gamma,\theta)}(z-x), \quad  0<s<t\le T,$$
 we obtain
\begin{align}
  \int_0^t \int_{\Rd} g_t^{(x,y)}(s,z)\, &dzds
	\leq  c_1 \int_0^t s^{-1+\epsilon/\alpha} (t-s)^{-1+(\theta-\epsilon)/\alpha} \notag
	\\
&\quad \cdot \big[ H_t^{(\gamma,\theta)}(y-x)+ H_s^{(\gamma,\theta)}(y-x)\big]\, ds  \leq  c_2 t^{-1+\theta/\alpha} H_t^{(\gamma,\theta)}(y-x) \label{I1del} \\
	&      +	c_3 \int_0^t s^{-1+(\epsilon-d)/\alpha} (t-s)^{-1+(\theta-\epsilon)/\alpha}
	            \left(1\vee \frac{|y-x|}{s^{1/\alpha}}\right)^{-\gamma-\alpha+\theta}\, ds \notag.
\end{align}
We need to carefully estimate the integral. We get
\begin{align*}
\Big( \int_0^{|x-y|^\alpha\wedge t} &+  \int_{|x-y|^\alpha\wedge t }^t \Big) s^{-1+(\epsilon-d)/\alpha} (t-s)^{-1+(\theta-\epsilon)/\alpha}
	            \left(1\vee \frac{|y-x|}{s^{1/\alpha}}\right)^{-\gamma-\alpha+\theta}\, ds\\
&= J_1+J_2.
\end{align*}
For $J_1$  after changing variables we get
\begin{align*}
J_1&=  t^{-1+ \frac{\gamma+\alpha-d}{\alpha}}  |x-y|^{-\alpha-\gamma+\theta} \int_0^{\frac{|x-y|^\alpha}{t}\wedge 1} \tau^{-1+\frac{\epsilon-d-\theta+\alpha+\gamma}{\alpha}}(1-\tau)^{-1+\frac{\theta-\epsilon}{\alpha}}d\tau.
\end{align*}
Treating two cases $|x-y|\leq (t/2)^{1/\alpha} $ and $|x-y|>  (t/2)^{1/\alpha}$   separately we get
\begin{align*}
  J_1 & \leq C t^{-1+ \frac{\gamma+\alpha-d}{\alpha}}  |x-y|^{-\alpha-\gamma+\theta}
	      \Big( \Big( \frac{|x-y|^\alpha}{t}\Big)^{\frac{\epsilon-d-\theta+\alpha+\gamma}{\alpha}}
				\I_{\{|x-y|\leq t^{1/\alpha}\}}+\I_{\{|x-y|> t^{1/\alpha}\}}\Big)  \\
      &  = C t^{-1+\frac{\theta}{\alpha}}\Big(\frac{t^{ \frac{-\epsilon}{\alpha}}}{|x-y|^{d-\epsilon}}
			  \I_{\{|x-y|\leq t^{1/\alpha}\}} +\frac{t^{ \frac{\gamma+\alpha-d-\theta}{\alpha}}  }{|x-y|^{\alpha+\gamma-\theta}} \I_{\{|x-y|> t^{1/\alpha}\}}\Big)  \\
      & =: C t^{-1+\frac{\theta}{\alpha}}K_t^{(1)}(x,y).
\end{align*}
For $J_2$ we have
\begin{align*}
 \int_{|x-y|^\alpha }^t s^{-1+(\epsilon-d)/\alpha} &(t-s)^{-1+(\theta-\epsilon)/\alpha} \I_{\{ |x-y|\leq s^{1/\alpha}\}}ds\\
	&    \leq   C t^{-1+ \frac{\theta-d}{\alpha}} \I_{\{ |x-y|\leq t^{1/\alpha}\}}
=: C t^{-1+\theta/\alpha} K_t^{(2)}(x,y).
\end{align*}
 Thus,
\begin{equation*}\label{I1-est}
I_1(\delta)\leq  C   t^{-1+\theta/\alpha} \big[H_t^{(\gamma,\theta)}(y-x) + K_t^{(1)}(x,y)+ K_t^{(2)}(x,y)\big].
\end{equation*}
For later convenience we note that
\begin{equation*}\label{K12}
\int_\rd [H_t^{(\gamma,\theta)}(y-x)+K_t^{(1)}(x,y)+ K_t^{(2)}(x,y)] dy \leq C, \quad x\in \rd, \  t\in (0,T].
\end{equation*}
To estimate the integrand in $I_2(\delta)$ we use Lemma~\ref{hol10}, Lemma~\ref{l:StGenerEst}, \eqref{Psi1} and \eqref{GH}:
\begin{align*}
  \Big| L^{x,\delta} &(p_{t-s}^z - p_{t-s}^x) (z-x)\, \Psi_s(x,y) \Big|
	 \leq  \mathcal{A}_{t-s}^\# (p_{t-s}^z - p_{t-s}^x) (z-x) \, \Psi_s(x,y) \\
	& \leq c_1 (|z-x|^\eta \wedge 1) (t-s)^{-1} G_{t-s}^{\alpha+\gamma-\theta}(z-x) s^{-1+\theta/\alpha}H_s^{(\gamma,\theta)}(y-x) \\
	& \leq c_2 s^{-1+\theta/\alpha} (t-s)^{-1+\theta/\alpha} H_{t-s}^{(\gamma-\theta,\theta)}(z-x) H_s^{(\gamma,\theta)}(y-x)  =:  h_t^{(x,y)}(s,z).
\end{align*}
Using Proposition~\ref{sub-conv} and  the same argument as for estimating  \eqref{I1del},  we get
\begin{align*}
  \int_0^t \int_{\Rd}& h_t^{(x,y)}(s,z)\, dz ds
	 \leq c_1 \int_0^t s^{-1+\theta/\alpha} (t-s)^{-1+\theta/\alpha} H_s^{(\gamma,\theta)}(y-x) \, ds \label{del2-1}\\
	& \leq  c_1 \int_0^t s^{-1+(\theta-d)/\alpha} (t-s)^{-1+\theta/\alpha}
	         \left(1 \vee \frac{|y-x|}{s^{1/\alpha}}\right)^{-\gamma-\alpha+\theta} \, ds \notag\\
&\leq c_3 t^{-1+\theta/\alpha}  K^{(3)}_t(x,y),
\end{align*}
where
\begin{equation*}\label{K3}
K_t^{(3)}(x,y)= \Big(t^{\frac{\theta-d}{\alpha}} + \frac{1}{|x-y|^{d-\theta}}\Big)\I_{\{|x-y|\leq t^{1/\alpha}\}} +
 \frac{t^{\frac{\gamma+\alpha-d}{\alpha}}}{|x-y|^{\alpha+\gamma-\theta}}\I_{\{|x-y|> t^{1/\alpha}\}}.
 \end{equation*}
 We also observe that
\begin{align*}
  \int_\rd  K_t^{(3)}(x,y) dy \leq  C t^{\theta/\alpha}, \quad x\in \rd,\ t\in (0,T].
\end{align*}
We get
\begin{equation*} \label{I2-est}
I_2(\delta)\leq  C t^{-1+\theta/\alpha}  K^{(3)}_t(x,y),\quad x,y\in \rd,\ t\in (0,T].
\end{equation*}
Furthermore, we have
\begin{equation*}\label{L-10}
\begin{split}
   &\int_{\Rd} \int_{|u|>\delta} \left| p_{t-s}^x (z-x+u)-p_{t-s}^x(z-x) \right| \, \nu(x,du) dz  \\
	& \leq  \int_{|u|>\delta} \int_{\Rd} \left( G_{t-s}^{(\alpha+\gamma)}(z-x+u) + G_{t-s}^{(\alpha+\gamma)}(z-x) \right) \, dz \nu(x,du) \\
	& \leq   c \nu_0(B(0,\delta)^c),
\end{split}
\end{equation*}
hence by Fubini's theorem we get for every $\delta>0$ that
\begin{equation*}\label{I3-est}
\begin{split}
  I_3(\delta)
	 & =   \int_0^t \int_{|u|>\delta} \left[ \int_{\Rd} \left( p_{t-s}^x (z-x+u)-p_{t-s}^x(z-x) \right)\, dz \right] \, \nu(x,du)\Psi_s(x,y)\, ds \\
	 & =   \int_0^t \int_{|u|>\delta} (1-1) \, \nu(x,du) \Psi_s(x,y)\, ds = 0.
\end{split}
\end{equation*}
Thus, by the dominated convergence theorem
\begin{eqnarray*}
  \lim_{\delta\to 0^+} L^{\delta} f_t^y(x)
	&   =  & \lim_{\delta\to 0^+} (I_1(\delta) + I_2(\delta) + I_3(\delta)) \\
	&   =  & \int_0^t \int_\rd L_x p_{t-s}^0 (x,z) \big[\Psi_s (z,y)- \Psi_s(x,y)\big] \, dz ds \\
	&      & +\,\, \int_0^t \int_\rd L^{x}  (p_{t-s}^z - p_{t-s}^x) (z-x) \,dz\Psi_s(x,y)\,ds, \\
	&   =  & \int_0^t \int_\rd \left[ L_x p_{t-s}^0 (x,z)\Psi_s (z,y)
	         - L^{x} p_{t-s}^x (z-x)\Psi_s(x,y) \right] \,dz ds.
\end{eqnarray*}
By Corollary~\ref{Lanihil} we also have
$  \int_{\Rd} L^{x} p_{t-s}^x (z-x)\, dz = 0$,
and \eqref{ltil} follows for every $t\in (0,T]$; since $T$ is arbitrary it holds for
every $t>0$.
\end{proof}
\begin{corollary}\label{rem-6-1}
There is a kernel $K_t(x,y)\ge 0$ and for every $T>0$ there is a constant $C$ such that $\int_\rd K_t(x,y)dy\leq C$
for all $t\in (0,T]$, $x\in \Rd$, and
\begin{equation}\label{L3-est}
\begin{split}
  |L_x^\delta p_t(x,y)| \leq  Ct^{-1} \Big(G_t^{\alpha+\gamma}(y-x) + t^{\theta/\alpha}  K_t(x,y)\Big),
\end{split}
\end{equation}
for all $t\in (0,T]$, $x,y\in \Rd$, $\delta>0$.
Furthermore,
\begin{equation}\label{Ldel1}
\left|\int_\rd L_x^\delta p_t(x,y)dy\right|\leq C t^{-1+(\eta\wedge \theta)/\alpha}, \quad t\in (0,T], \ x\in \Rd, \ \delta>0.
\end{equation}
\end{corollary}
\begin{proof}
Let $K=K^{(1)}+K^{(2)}+K^{(3)}+H^{(\gamma,\theta)}$, where the terms on the right-hand side are as in
the proof of Lemma~\ref{lem-ltil}. This gives \eqref{L3-est}.
 To prove \eqref{Ldel1} we consider
\begin{align*}
\Big| \int_\rd L_x^\delta p_t^0(x,y)\, dy\Big|  &\leq  \Big| \int_\rd  L_x^\delta [ p_t^y(y-x)- p_t^x(y-x)] dy \Big| +\Big|  \int_\rd L_x^\delta p_t^x(y-x)dy\Big|.
\end{align*}
The last integral is $0$. To estimate the first integral we use the H\"older continuity of $p_x^z(y-x)$ in $z$, i.e. \eqref{hol10-eq} with $\beta=0$ and $|\beta|=2$:
\begin{align*}
\big| p_t^y(y-x)- p_t^x(y-x)\big| &\leq c |x-y|^\eta t^{-|\beta|/\alpha} G_t^{(\alpha+\gamma-\theta)} (y-x)
\\
&\leq  c t^{(|\beta|-\eta)/\alpha} G_t^{(\alpha+\gamma-\theta-\eta)} (y-x).
\end{align*}
Applying Lemma~\ref{l:StGenerEst} with $\kappa=\theta+\eta$,
$\zeta=d$, we get
$$
\int_\rd \big|L_x^\delta [ p_t^y(y-x)- p_t^x(y-x)] \big| dy  \leq  C t^{-1+\eta/\alpha} \int_\rd G_t^{(\alpha+\gamma-\theta-\eta)} (y-x)dy\leq  C t^{-1+\eta/\alpha}.
$$
The proof is complete.
\end{proof}

We next show how to differentiate $(p^0\boxtimes \Psi)_t(x,y)$ in $t$.
\begin{lemma}\label{der2} For $x,y\in\Rd$, $0<s<t$, $0<t\le T$, we have
\begin{align}
\label{dif1}
&  \partial_t \int_\rd p_{t-s}^0(x,z)\Psi_s(z,y)\, dz=\int_\rd   \partial_t p_{t-s}^0(x,z)\Psi_s(z,y)\, dz,\\
\label{dif21}
  &\int_s^t \left| \int_\rd \partial_r p_{r-s}^0(x,z)\Psi_s (z,y) \, dz \right|\, dr<\infty,\\
&\label{dif2}
  \int_0^t \int_0^r \left| \int_\rd \partial_r p_{r-s}^0(x,z)\Psi_s (z,y) \, dz \right|\, dsdr<\infty.
\end{align}
\end{lemma}

\begin{proof}
In order to prove \eqref{dif1} it suffices to show that for all fixed $t>s>0$ and $x,y\in\Rd$ there is $\varepsilon_0>0$ and
a function
$g(z)\ge 0$ such that $\int_{\Rd} g(z)\, dz < \infty $ and
$$
  \left| \partial_t p_{t+\varepsilon-s}^0(x,z)\Psi_s(z,y) \right| \leq g(z), \quad z\in\Rd,\varepsilon\in (-\varepsilon_0, \varepsilon_0).
$$
Using Lemma~\ref{l:dertxpest} and \eqref{Psi1}, for every $\varepsilon_0<t-s$ we get
\begin{align*}
         \big| \partial_t &p_{t+\varepsilon-s}^0(x,z) \Psi_s(z,y)  \big|
          \leq c_1 (t+\varepsilon-s)^{-1} G_{t+\varepsilon-s}^{(\alpha+\gamma)}(z-x) s^{-1+\theta/\alpha}
				        e^{c_2 s} H_s^{(\gamma,\theta)}(z-y) \\
         & \leq c_1 (t-\varepsilon_0-s)^{-1-d/\alpha} s^{-1+\theta/\alpha} e^{c_2 s} H_s^{(\gamma,\theta)}(z-y) : = g(z),
\end{align*}
and the finiteness of $\int_{\Rd} g(z)\, dz$ follows from Proposition~\ref{sub-conv}.

The integral in \eqref{dif2} is not bigger than
\begin{align*}
&\int_0^t \int_0^r \left| \int_\rd \partial_r  p_{r-s}^0(x,z)(\Psi_s (z,y)-\Psi_s(x,y))\, dz \right| \, ds dr\\
  & + \int_0^t \int_0^r \left| \int_\rd \partial_r p_{r-s}^0(x,z)\, dz \right| |\Psi_s (x,y)| \,ds dr
	  = I_1+I_2.
\end{align*}
For $I_1$ by Lemma~\ref{l:dertxpest}, Lemma~\ref{phol}, \eqref{GH} and Proposition~\ref{sub-conv} we get that for every $\epsilon\in (0,\theta)$ and
$\theta\in (0,\eta\wedge \alpha \wedge (\gamma-d+\alpha))$,
\begin{eqnarray*}
  I_1
	& \leq & c_1 \int_0^t \int_0^r \int_\rd (r-s)^{-1} G_{r-s}^{(\alpha+\gamma)}(z-x)
	         \big( |x-z|^{\theta-\epsilon}\wedge 1\big) s^{-1+\epsilon/\alpha} \\
  &      & \cdot \big[H_s^{(\gamma,\theta)}(y-z)+ H_s^{(\gamma,\theta)}(y-x)\big]\, dz\,ds\, dr \\
  & \leq & c_1 \int_0^t \int_0^r \int_\rd (r-s)^{-1+(\theta-\epsilon)/\alpha} s^{-1+\epsilon/\alpha}
	         H_{r-s}^{(\gamma, \theta-\epsilon)}(z-x)\\
  &      & \cdot \big[H_s^{(\gamma,\theta)}(y-z)+ H_s^{(\gamma,\theta)}(y-x)\big]\, dz\,ds\, dr \\
	& \leq & c_2 \int_0^t \int_0^r \int_\rd (r-s)^{-1+(\theta-\epsilon)/\alpha} s^{-1+\epsilon/\alpha}
	         H_{r-s}^{(\gamma, \theta)}(z-x)\\
  &      & \cdot \big[H_s^{(\gamma,\theta)}(y-z)+ H_s^{(\gamma,\theta)}(y-x)\big]\, dz\,ds\, dr.
	         	\end{eqnarray*}
Further,
	         \begin{eqnarray*}	
	I_1
	& \leq & c_3 \int_0^t \int_0^r (r-s)^{-1+(\theta-\epsilon)/\alpha} s^{-1+\epsilon/\alpha}
	         \big[ H_{r}^{(\gamma, \theta)}(y-x)+ H_s^{(\gamma,\theta)}(y-x) \big] \, dsdr\\
	         &=  & c_3 \int_0^t \int_s^t (r-s)^{-1+(\theta-\epsilon)/\alpha} s^{-1+\epsilon/\alpha}
	         \big[ H_{r}^{(\gamma, \theta)}(y-x)+ H_s^{(\gamma,\theta)}(y-x) \big] \, drds\\
  &   =  & c_4 \Big[
	         \int_0^t \int_s^t (r-s)^{-1+(\theta-\epsilon)/\alpha} s^{-1+\epsilon/\alpha}
	         H_{r}^{(\gamma, \theta)}(y-x) \, drds  \\
	&      & \,\, + \int_0^t (t-s)^{(\theta-\epsilon)/\alpha}s^{-1+\epsilon/\alpha}H_s^{(\gamma,\theta)}(y-x)\, ds \Big].
	\end{eqnarray*}
By the estimate $H_t^{(\gamma,\theta)}(x)\leq c |x|^{-\alpha-\gamma+\theta} t^{1+(\gamma-\theta-d)/\alpha}$ we obtain
\begin{align*}
  I_1
	& \leq  c_4 \Big[  \int_0^t \int_s^t (r-s)^{-1+(\theta-\epsilon)/\alpha} s^{-1+\epsilon/\alpha}
	         r^{-d/\alpha}\left(\frac{|y-x|}{r^{1/\alpha}}\right)^{-\gamma-\alpha+\theta} \, drds \\
	&     \quad   + 	\int_0^t (t-s)^{(\theta-\epsilon)/\alpha}s^{-1+(\epsilon-d)/\alpha}\left(\frac{|y-x|}{s^{1/\alpha}}\right)^{-\gamma-\alpha+\theta}\, ds
					\Big]\\
  &   =   c_5 |y-x|^{-\alpha-\gamma+\theta} \Big[  t^ {1+(\gamma-d-\theta)/\alpha} \int_0^t (t-s)^{(\theta-\epsilon)/\alpha}s^{-1+\epsilon/\alpha}\, ds
	        \\
&\quad  +  \int_0^t (t-s)^{(\theta-\epsilon)/\alpha}s^{(\gamma-\theta+\epsilon-d)/\alpha}\, ds \Big]
	  \leq   c_6 |y-x|^{-\alpha-\gamma+\theta} t^{1+(\gamma-d)/\alpha}.
\end{align*}
We note that the constants $c_i$ in this proof may depend on $T$.
For $I_2$ by Lemma~\ref{derp}, \eqref{Psi1} and \eqref{GH} for $\theta\in (0,\eta\wedge \frac{\alpha+\gamma-d}{2})$ we get similarly
\begin{align*}
  I_2
	& \leq c_1 \int_0^t \int_0^r (r-s)^{-1+\theta/\alpha} |\Psi_s (x,y)| \, ds dr \\
  & \leq c_2 \int_0^t \int_0^r (r-s)^{-1+\theta/\alpha} s^{-1+\theta/\alpha} H_s^{(\gamma,\theta)} (y-x) \, ds dr  \\
	&   =  c_2 \int_0^t \int_s^t (r-s)^{-1+\theta/\alpha} s^{-1+\theta/\alpha} H_s^{(\gamma,\theta)} (y-x) \, dr ds  \\
  & \leq c_3 |y-x|^{-\alpha-\gamma+\theta} \int_0^t (t-s)^{\theta/\alpha} s^{(\gamma-d)/\alpha}ds
     = c_4 |y-x|^{-\alpha-\gamma+\theta} t^{1+(\theta+\gamma-d)/\alpha},
\end{align*}
because we assumed $\gamma-d+\alpha>0$.
This yields \eqref{dif21} and \eqref{dif2}.
\end{proof}

\subsection{Proof of Theorem~\ref{t-exist}}
By \eqref{r}, Lemma~\ref{th-ptx2}, \eqref{e:small}, Lemma~\ref{delta} and \eqref{H30} we obtain \eqref{f-sol}.
We next verify \eqref{eq:pofs}.
Using Lemma~\ref{der2},  \ref{phol} and \ref{delta} we get
$$
  \int_s^t \Big[\partial_r \int_\rd p_{r-s}^0(x,z)\Psi_s(z,y)\,dz\Big]\,dr
	 = \int_\rd p_{t-s}^0(x,z)\Psi_s(z,y)\, dz- \Psi_s(x,y).
$$
Integrating the above equation from $0$ to $t$ and using Lemma~\ref{der2}
and Fubini's theorem we obtain
\begin{equation*}
  \int_0^t \int_\rd p_{t-s}^0(x,z)\Psi_s(z,y)\, dzds - \int_0^t \!\! \Psi_s(x,y)\, ds
  \! = \!\! \int_0^t \int_0^r \int_\rd  \!\! \partial_r  p_{r-s}^0(x,z)\Psi_s (z,y)\, dz dsdr.
\end{equation*}
By Corollary~\ref{Phik_t_cont} and the locally uniform convergence of the series defining $\Psi_t$ the function $t\to \Psi_t(x,y)$ is continuous, therefore
\begin{equation}\label{par10}
\partial_t p_t(x,y)=  \partial_t p^0_t(x,y)+ \Psi_t(x,y)+ \int_0^t \int_\rd \partial_t p_{t-s}^0(x,z)\Psi_s(z,y)\, dzds.
\end{equation}
Subtracting $L_x p_t(x,y)$ from both sides  and using Lemma~\ref{lem-ltil} we get
\begin{align*}
\big(\partial_t -L_x\big)  p_t(x,y)&= -\Phi_t(x,y)+ \Psi_t(x,y)- \int_0^t \int_\rd\Phi_{t-s}(x,z)\Psi_s(z,y)\, dzds=0.
\end{align*}
The proof of Theorem~\ref{t-exist} is complete.\qed{}

\section{
Further regularity
}\label{time}

\subsection{Time derivatives of $\Psi_t(x,y)$}

In view of the definition of $p_t(x,y)$, to study its regularity in time we begin with an auxiliary estimate of the time derivative of $\Psi_t(x,y)$.

\begin{lemma}\label{phi-der}
The function $\Psi_t(x,y)$ is differentiable in $t$ and $\partial_t \Psi_t(x,y)$ is  continuous 
 on $(0,\infty)$. There are $C,c>0$ and $\theta\in(0,\alpha\wedge \eta \wedge (\alpha+\gamma-d))$ such that
\begin{equation}\label{phi-der10}
  \big|\partial_t \Psi_t(x,y)\big|
  \leq C e^{ct} t^{-2+\theta/\alpha} H_t^{(\gamma,\theta)}(y-x), \quad x,y\in\Rd,\, t>0.
\end{equation}
\end{lemma}
\begin{proof}
It follows from Lemma~\ref{Phi-up} and \eqref{GH} that $\partial_t \Phi_t(x,y)$  is  continuous and
\begin{equation} \label{der-Phi1}
  \big|  \partial_t \Phi_t(x,y)\big|\leq C_\Phi  t^{-2+\theta/\alpha} H_t^{(\gamma,\theta)}(y-x), \quad x,y\in\Rd,\, t>0.
\end{equation}
We will show by induction for all $k\geq 1$ that $\partial_t \Phi_t^{\boxtimes k}=\partial_t (\Phi_t^{\boxtimes k})$ exists
and
\begin{equation}\label{der-Phik}
  \big|  \partial_t \Phi_t^{\boxtimes k}(x,y)\big|\leq \frac{C_3C_4^k}{\Gamma(k\theta /\alpha)}   t^{-2+\theta k/\alpha} H_t^{(\gamma,\theta)}(y-x), \quad x,y\in\Rd,\, t>0,
\end{equation}
where
\begin{equation}\label{C34}
C_3=(1\vee  (\Gamma(\theta/\alpha))^{-1})C_1, \quad C_4 = 8    \big( 1\vee  (2-2\theta/\alpha)^{-\theta/\alpha})C_2,\
\end{equation}
and $C_1$, $C_2$  come  from \eqref{Phi20}. The case of $k=1$ is verified by \eqref{Phi1}.
Note that
\begin{equation*}\label{Phik}
  \begin{split}
    \Phi^{\boxtimes (k+1)}_t(x,y)
		 =\, \int_0^{t/2}\int_{\rd}\Phi_{t-s}^{\boxtimes k}(x,z)\Phi_s(z,y)\,dz ds
    +\int_0^{t/2}\int_{\rd}\Phi_{s}^{\boxtimes k}(x,z)\Phi_{t-s}(z,y)\,dz ds,
\end{split}
\end{equation*}for $k\in \N$.
Accordingly, we claim that for $k\in \N$,
\begin{equation}\label{Phik10}
\begin{split}
&\partial_t\Phi^{\boxtimes (k+1)}_t(x,y)=\int_0^{t/2}\int_{\rd}\partial_t\Phi^{\boxtimes k}_{t-s}(x,z)\Phi_s(z,y)\,dz ds\\
&\quad +
\int_0^{t/2}\int_{\rd}\Phi_{s}^{\boxtimes k}(x,z)\partial_t\Phi_{t-s}(z,y)\,dz ds
\quad +\int_{\rd}\Phi_{t/2}^{\boxtimes k}(x,z)\Phi_{t/2}(z,y)\, dz.
\end{split}
\end{equation}
Indeed, we consider $\left(\Phi^{\boxtimes (k+1)}_{t+h}(x,y)-\Phi^{\boxtimes (k+1)}_t(x,y)\right)/h$ as $h\to 0$.
If for some $k\geq 1$, continuous $\partial_t \Phi^{\boxtimes k}_t(x,y)$ exists for all $t>0,\, x,y\in\Rd,$ and
\eqref{der-Phik} holds for every $t>0$, then for $h\in (-t/4,t/4)$ we have
\begin{align*}
  \left| \partial_t\Phi^{\boxtimes k}_{t+h-s}(x,z)\Phi_s(z,y)\right|
	& \leq  c_1 (t+h-s)^{-2+k\theta/\alpha} s^{-1+\theta/\alpha} H_{t+h-s}^{(\gamma,\theta)}(z-x) H_{s}^{(\gamma,\theta)}(y-z) \\
	& \leq  c_2 (t-s)^{-2+k\theta/\alpha} s^{-1+\theta/\alpha} H_{t-s}^{(\gamma,\theta)}(z-x) H_{s}^{(\gamma,\theta)}(y-z)  \\
	&   =:  g_t^{(x,y)}(s,z), \qquad  s\in (0,t/2), x,y,z\in\Rd.
\end{align*}
It follows from Proposition~\ref{sub-conv}  that $\int_0^{t/2} \int_{\Rd} g_t^{(x,y)}(s,z)\, dzds<\infty$. 
Estimating similarly $\left| \Phi_{s}^{\boxtimes k}(x,z)\partial_t \Phi_{t+h-s}(z,y)\right|$, by the continuity of
$t\mapsto \Phi_{t}^{\boxtimes k}(x,z)$
we get \eqref{Phik10}.
Let $I_1$, $I_2$, $I_3$ be the integrals in \eqref{Phik10}, respectively.
Using induction, Lemma~\ref{Phi-up} and Proposition~\ref{sub-conv} we get for the first term
\begin{eqnarray*}
  |I_1|
	& \leq & \frac{2C_\Phi C_3C_4^k  t^{-1}}{\Gamma(k\theta/\alpha)} \int_0^t \int_\Rd (t-s)^{-1+k\theta/\alpha} s^{-1+\theta/\alpha}
	         H_{t-s}^{(\gamma,\theta)}(z-x) H_{s}^{(\gamma,\theta)}(y-z)\, dzds \\
  & \leq & \frac{2C_3 C_\Phi C_{H}C_4^k }{\Gamma(k\theta/\alpha)}
            B\big(k\theta/\alpha,\theta/\alpha\big) t^{-2+(k+1)\theta/\alpha} H_{t}^{(\gamma,\theta)}(y-x) \\
  &   =  & \frac{C_3}{\Gamma((k+1)\theta/\alpha)}  \cdot \big(2 C_2C_4^k\big) \cdot  t^{-2+(k+1)\theta/\alpha} H_{t}^{(\gamma,\theta)}(y-x).  
\end{eqnarray*}
The same estimate holds for  $I_2$, so let us estimate $I_3$.  By \eqref{Phi20},
\begin{eqnarray*}
  |I_3|
	& \leq & \frac{C_1 C_\Phi C_2^k }{\Gamma(k\theta/\alpha)} \Big(\frac{t}{2}\Big)^{-2+(k+1)\theta/\alpha}\int_\Rd
	         H_{t/2}^{(\gamma,\theta)}(z-x) H_{t/2}^{(\gamma,\theta)}(y-z)\, dz \\
  & \leq & \frac{C_1 C_\Phi   C_{H} C_2^k}{\Gamma(k\theta/\alpha)} \Big(\frac{t}{2}\Big)^{-2+(k+1)\theta/\alpha} H_{t}^{(\gamma,\theta)}(y-x).
\end{eqnarray*}
Using the inequality $u\leq e^u$,  valid for all $u\in \real$, we get for $\zeta=\theta/\alpha$,
 $$
   \Gamma((k+1)\zeta) = \int_0^\infty e^{-u}u^{(k+1)\zeta-1}\, du \leq (1-\zeta) ^{-k\zeta} \Gamma(k \zeta).
 $$
Therefore,
 \begin{eqnarray*}
   |I_3|
	& \leq & \frac{C_1 C_\Phi   C_{H} C_2^k2^{2-(k+1)\theta/\alpha}}{(1-\theta/\alpha)^{k\theta/\alpha}\Gamma((k+1)\theta/\alpha)} t^{-2+(k+1)\theta/\alpha} H_{t}^{(\gamma,\theta)}(y-x) \\
	& \leq & \frac{C_3}{\Gamma((k+1)\theta/\alpha)} 4\Big(\frac{C_2}{(2-2\theta/\alpha)^{\theta/\alpha}} \Big)^{k+1}
 t^{-2+(k+1)\theta/\alpha} H_{t}^{(\gamma,\theta)}(y-x),
 \end{eqnarray*}
because $C_2 = C_\Phi C_{H} \Gamma(\theta/\alpha)$. 
 Observe that for $C_4 $  given in \eqref{C34} we have 
 $$
  4 C_2 C_4^k+ 4 \Big(\frac{C_2}{(2-2\theta/\alpha)^{\theta/\alpha})} \Big)^{k+1}\leq C_4^{k+1}, 
 $$
and so
 $$
 I_1+ I_2+I_3 \leq \frac{C_3 C_k^{k+1}}{\Gamma((k+1)\theta/\alpha)}
 t^{-2+(k+1)\theta/\alpha} H_{t}^{(\gamma,\theta)}(y-x),
 $$
proving  \eqref{der-Phik}.   By \eqref{eq:Get}
and \eqref{der-Phik} we get
\eqref{phi-der10}.
\end{proof}

\subsection{Proof of Theorem~\ref{t-prop}}
The proof of \eqref{up100} for $k=0$ easily follows from Proposition~\ref{pr1}. Let us show \eqref{up100} for $k=1$. Our starting point is \eqref{r}.
Lemma~\ref{l:dertxpest} estimates $\partial_t p_t^0(x,y)$.
We then use the estimate for $\partial_t \Psi_t(x,y)$ given in \ref{phi-der}. The estimate of $\partial_t \big(p^0 \boxtimes \Psi\big)_t(x,y)$  can be obtained similarly as the estimates for $\partial_t \Phi_t(x,y)$ in Lemma~\ref{phi-der}.  Indeed, as in the proof of \eqref{Phik10}, using \eqref{GH22} for every $h\in (-t/4,t/4)$ we get
\begin{eqnarray*}
  \left|\partial_t p^0_{t+h-s}(x,z)\Psi_s(z,y)\right|
	& \leq & c_1 (t-s)^{-1}G_{t-s}^{(\alpha+\gamma)}(z-x) s^{-1+\theta/\alpha}e^{c_2s} H_s^{(\gamma,\theta)}(y-z) \\
	& \leq & c_1 \frac{(t-s)^{-1}}{t^{-\theta/\alpha}\wedge 1}H_{t-s}^{(\gamma,\theta)}(z-x) s^{-1+\theta/\alpha}e^{c_2s} H_s^{(\gamma,\theta)}(y-z) \\
	&  =:  & g^{(x,y)}_t(s,z), \quad s\in (0,t), 
\end{eqnarray*}
and it follows from Proposition~\ref{sub-conv} that the majorant satisfies
$$
\int_0^{t/2}\int_{\Rd} g^{(x,y)}_t(s,z)\, dsdz < c t^{-1+\theta/\alpha} e^{c_3t} H_t^{(\gamma,\theta)}(y-x)<\infty.
$$
 Similarly we estimate
$|p^0_{s}(x,z)(\partial_t\Psi)_{t+h-s}(z,y)|$. These bounds and the continuity of $t\mapsto p^0_t(x,y)$ and $t\mapsto \Psi_t(x,y)$ allow us to write
\begin{equation*}
\begin{split}
\partial_t \big(p^0& \boxtimes \Psi\big)_t(x,y)=\int_0^{t/2}\int_{\rd}(\partial_t p^0)_{t-s}(x,z)\Psi_s(z,y)\,dz ds\\
&\quad +
\int_0^{t/2}\int_{\rd}p^0_{s}(x,z)(\partial_t\Psi)_{t-s}(z,y)\,dz ds +\int_{\rd}p^0_{t/2}(x,z)\Psi_{t/2}(z,y)\, dz.
\end{split}
\end{equation*}
We obtain
$$
\big|\partial_t \big(p^0 \boxtimes \Psi\big)_t(x,y)\big| \leq C t^{-1} e^{ct} G_t^{(\alpha+\gamma)}(y-x).
$$
This
finishes the proof of \eqref{up100}.
We next prove \eqref{ptx-hol}. We observe that by
\eqref{lip-p},
\begin{equation*}\label{hol-p0}
  \left| p_t^0(x_1,y)-p_t^0(x_2,y)\right| \leq C \Big( \frac{|x_1-x_2|}{t^{1/\alpha}} \wedge 1\Big)
	\left( G_t^{(\alpha+\gamma)}(y-x_1)+G_t^{(\alpha+\gamma)}(y-x_2)\right).
\end{equation*}
We first suppose that $t\in (0,1]$. Using  \eqref{Psi1} for $\Psi$, \eqref{GH22} and the sub-convolution property of $H_t^{(\gamma,\theta)}(x)$ we obtain
\begin{align*}
  \int_0^t  &\int_\rd |p_{t-s}^0(x_1,z)-p_{t-s}^0(x_2,z)||\Psi_s(z,y)|\, dzds \\
  & \leq c_1 |x_2-x_1|^\theta \int_0^t \int_\rd (t-s)^{-\theta/\alpha} \Big( G_{t-s}^{(\alpha+\gamma)}(z-x_1)+G_{t-s}^{(\alpha+\gamma)}(z-x_2)\Big) \\
  & \quad \cdot s^{-1+\theta/\alpha} H_s^{(\gamma,\theta)}(y-z)\, dzds \\
  & \leq c_2 |x_2-x_1|^\theta  \left( H_t^{(\gamma,\theta)}(y-x_1)+ H_t^{(\gamma,\theta)}(y-x_2)\right) \\
  & \leq c_2 \Big(\frac{|x_2-x_1|}{t^{1/\alpha}}\Big)^\theta \left( G_t^{(\alpha+\gamma)}(y-x_1)+ G_t^{(\alpha+\gamma)}(y-x_2)\right),
\end{align*}
where in the last line we used that $t^{\theta/\alpha} H_t^{(\gamma,\theta)}(x)\leq G_t^{(\alpha+\gamma)}(x)$ for $t\in (0,1]$.
Next we suppose that $t>1$. Using the estimate for $\Psi$  and \eqref{GH22} twice  we get
\begin{align*}
  \int_0^t
	& \int_\rd |p_{t-s}^0(x_1,z)-p_{t-s}^0(x_2,z)| |\Psi_s(z,y)|\, dzds\\
	& \leq c_1 |x_2-x_1|^\theta e^{ct} \int_0^{t} (t-s)^{-\theta/\alpha}\int_\rd \left( G_{t-s}^{(\alpha+\gamma)}(z-x_1)+G_{t-s}^{(\alpha+\gamma)}(z-x_2)\right) \\
  & \quad \cdot s^{-1+\theta/\alpha} H_s^{(\gamma,\theta)}(y-z)\, dzds \\
	& \leq c_1 |x_2-x_1|^\theta e^{ct} \int_0^{t} (1\vee (t-s)^{-\theta/\alpha}) \int_\rd \left( H_{t-s}^{(\gamma,\theta)}(z-x_1)+H_{t-s}^{(\gamma,\theta)}(z-x_2)\right) \\
  & \quad \cdot s^{-1+\theta/\alpha} H_s^{(\gamma,\theta)}(y-z)\, dzds \\
   & \leq c_2|x_2-x_1|^\theta t^{\theta/\alpha}e^{ct}
    \left( H_t^{(\gamma,\theta)}(y-x_1)+H_t^{(\gamma,\theta)}(y-x_2)\right) \\
  & \leq c_2|x_2-x_1|^\theta   e^{ct}
    \left(G_t^{(\alpha+\gamma)}(y-x_1)+G_t^{(\alpha+\gamma)}(y-x_2)\right) \\
  & \leq c_2 \left(\frac{|x_2-x_1|}{t^{1/\alpha}}\right)^\theta   e^{c_3t}
    \Big(G_t^{(\alpha+\gamma)}(y-x_1)+G_t^{(\alpha+\gamma)}(y-x_2)\Big).
\end{align*}
This  ends the proof of \eqref{ptx-hol} for $t>0$.

According to our choice of the first approximation $p^0_t(x,y)$, the regularity of $y\mapsto p_t(x,y)$ is less obvious than that of $x\mapsto p_t(x,y)$. The next result gives a preparation for such regularity and may be confronted with Lemma~\ref{phol}.
\begin{lemma}\label{contofPhi}
  For all $t>0$ and $y\in\Rd$ we have
	\begin{equation}\label{Lpty_cont}
		\lim_{z\to y} \sup_{x\in\Rd} |\Phi_t (x,z) - \Phi_t (x,y)| = 0.
	\end{equation}
\end{lemma}
\begin{proof}
Since $\partial_t p_t^z(x) = L^z p_t^z(x)$ we get
	\begin{align*}
	  | \partial_t (p_t^z (z-x)& - p_t^y(y-x))|
		 =   | L^z p_t^z (z-x) - L^y p_t^y(y-x) | \\
		& \leq  | L^z p_t^z (z-x) - L^y p_t^z(z-x)| + | L^y p_t^z (z-x) - L^y p_t^y(z-x)| \\
		&       + \, | L^y p_t^y (z-x) - L^y p_t^y(y-x) | \\
		&  = I_1 + I_2 + I_3.
	\end{align*}
From \textbf{A2} and Lemma~\ref{l:StGenerEst} we have
	\begin{eqnarray*}
		 I_1
		 & \leq &  \int_\Rd \left| p_t^z(z-x+u) - p_t^z(z-x) - u \cdot \nabla_x p_t^z(z-x) \indyk{|u|\leq t^{1/\alpha}}(u)\right|
		           \left|\nu(z,du) - \nu(y,du) \right| \\
		 & \leq & M_0 (|z-y|^\eta\wedge 1) \modgener p_t^z (z-x) \\
		 & \leq & c_1 (|z-y|^\eta\wedge 1) t^{-1} G_t^{(\alpha+\gamma)}(z-x) \leq c_1 t^{-1-d/\alpha} (|z-y|^\eta\wedge 1).
	\end{eqnarray*}
	From Lemma~\ref{hol10} and Lemma~\ref{l:StGenerEst} we obtain
	\begin{eqnarray*}
	  I_2
		&  =   & |L^y (p_t^z - p_t^y)(z-x)| \leq \modgener (p_t^z-p_t^y)(z-x) \\
		& \leq & c_2 (|z-y|^\eta\wedge 1)t^{-1} G_t^{(\alpha+\gamma-\theta)}(z-x) \leq c_2 t^{-1-d/\alpha} (|z-y|^\eta\wedge 1).
	\end{eqnarray*}
	Finally, let $|y-z|<t^{1/\alpha}$ and $g_t(w) = p_t^y (w-(y-z)) - p_t^y(w) $.
	Using Taylor expansion and Lemma \ref{th-ptx2},
	for every $\beta\in\N_0^d$ such that $|\beta|\leq 2$ we get
	$$
	  |\partial^\beta_w g_t(w)|
		 \leq  c_3 |z-y|t^{(-1-|\beta|)/\alpha} G_t^{(\alpha+\gamma)}(w).
	$$
This and Lemma~\ref{l:StGenerEst} yield
	\begin{align*}
	  I_3
		&   =   | L^y g_t(y-x) | \leq  \modgener g_t(y-x)  \leq c_4 |z-y| t^{-1-1/\alpha} G_t^{(\alpha+\gamma)}(y-x) \\
		& \leq  c_4 t^{-1-(1+d)/\alpha} |z-y|.
	\end{align*}
	Therefore,
	$$
	  \lim_{z\to y} \sup_{x\in\Rd} |\partial_t (p_t^z (z-x) - p_t^y(y-x))| = 0.
	$$
	Similarly,
	\begin{eqnarray*}
	  |L^x p_t^z(z-x) - L^x p_t^y(y-x)|
		& \leq & |L^x (p_t^z - p_t^y)(z-x)| + |L^x p_t^y(z-x) - L^x p_t^y(y-x)| \\
		& \leq & \modgener (p_t^z - p_t^y)(z-x) + \modgener g_t(y-x) \\
		& \leq & c_5 t^{-1-d/\alpha}(|z-y|^\eta\wedge 1) + c_6 t^{-1-(1+d)/\alpha}|z-y|,
	\end{eqnarray*}
	$
	  \lim_{z\to y} \sup_{x\in\Rd} |L^x (p_t^z (z-x) - p_t^y(y-x))| = 0
	$,
	and \eqref{Lpty_cont} follows.
  \end{proof}
	\begin{lemma}\label{p_continuity}
	  For all $t>0$ and $x\in\Rd$ the function $y\mapsto p_t(x,y)$ is continuous.
	\end{lemma}
	\begin{proof}
For the proof we rely on \eqref{eq:pfo}.	
		It is straightforward to see that
		$$
		  H_s^{(\gamma,\theta)}(y+h) \leq c H_s^{(\gamma,\theta)}(y),
		$$
		where $y\in\Rd$, $s>0$, $h\in \Rd$ and $|h|<s^{1/\alpha}$.
		Let $T\in (0,\infty)$ and $t\in (0,T]$.
		By Theorem~\ref{t-prop}, \eqref{GH22}, \eqref{Phi1} and Proposition~\ref{sub-conv} for
		every $\varepsilon>0$ and $|h|<\varepsilon^{1/\alpha}$ we get
		\begin{eqnarray*}
		  &      & \int_0^{t-\varepsilon} \int_{\Rd} | p_s(x,z)\Phi_{t-s}(z,y+h)|\, dzds \\
			& \leq &  \frac{c_1e^{c_2t}}{t^{-\theta/\alpha}\wedge 1}
			          \int_0^{t-\varepsilon} \int_{\Rd} H_s^{(\gamma,\theta)}(z-x)
								(t-s)^{-1+\theta/\alpha} H_{t-s}^{(\gamma,\theta)}(y+h-z)\, dzds \\
			& \leq &  c_3 \int_0^{t-\varepsilon} \int_{\Rd} (t-s)^{-1+\theta/\alpha} H_s^{(\gamma,\theta)}(z-x)
								 H_{t-s}^{(\gamma,\theta)}(y-z)\, dzds \\
			& \leq &  c_4 \int_0^{t-\varepsilon} (t-s)^{-1+\theta/\alpha} H_t^{(\gamma,\theta)}(y-x)\, ds
			          \leq c_5 H_t^{(\gamma,\theta)}(y-x),
		\end{eqnarray*}
		with $c_3,c_4,c_5$ depending on $T$. By the dominated convergence
		and Lemma~\ref{contofPhi},
		$$
		  \lim_{h\to 0} \int_0^{t-\varepsilon} \int_{\Rd} p_s(x,z)\Phi_{t-s}(z,y+h)\, dzds
			= \int_0^{t-\varepsilon} \int_{\Rd}  p_s(x,z)\Phi_{t-s}(z,y)\, dzds.
		$$
		Furthermore, for every $|h|<t^{1/\alpha}$,
		\begin{eqnarray*}
		  &      & \left|\int_{t-\varepsilon}^{t} \int_{\Rd} p_s(x,z)\Phi_{t-s}(z,y+h) \, dzds \right| \\
			& \leq &  c_6
			          \int_{t-\varepsilon}^{t} \int_{\Rd} H_s^{(\gamma,\theta)}(z-x)
								(t-s)^{-1+\theta/\alpha} H_{t-s}^{(\gamma,\theta)}(y+h-z)\, dzds \\
			& \leq &  c_7 \int_{t-\varepsilon}^{t} (t-s)^{-1+\theta/\alpha} H_t^{(\gamma,\theta)}(y+h-x)\, ds
			          \leq c_8 \varepsilon^{\theta/\alpha} H_t^{(\gamma,\theta)}(y-x)<\infty.
		\end{eqnarray*}
		This and Lemma~\ref{pty_cont} yield the continuity of $y\mapsto p_t(x,y)$.
	\end{proof}
The proof of Theorem~\ref{t-prop} is complete.\qed{}

\section{The maximum principle}\label{Uni}

In this part of our development we follow Kochube\u\i{}'s argument from \cite[Section 6]{MR972089} with some modifications--we temper by $e^{-\lambda t}$ rather than restrict time. For $\lambda\in \RR$ we let $\tilde{p}_t(x,y) = e^{-\lambda t} p_t(x,y)$, where $t>0$, $x,y\in \Rd$.
By Theorem~\ref{t-prop}, $\tilde{p}_t(x,y) \leq C e^{-(\lambda-c)t} G_t^{\alpha+\gamma}(y-x)$.
We can give a solution to the Cauchy problem for $L-\lambda$.
\begin{lemma}\label{nn}
If $f\in C_0(\Rd)$, $u(t,x)=\int_\Rd \tilde p_t(x,y)f(y)\, dy$ for $t>0$ and $u(0,x)=f(x)$, $x\in \Rd$, then
$u$ is a continuous function on $[0,\infty)\times \Rd$, and
\begin{equation}\label{pt_L_tu}
  \big( \partial_t -L_x+\lambda\big) u(t,x)= 0,\quad t>0,\  x\in \Rd.
\end{equation}
If $\lambda\ge c$, where  $c$ is from {\rm Theorem~\ref{t-prop}}, then
$u\in C_0([0,\infty)\times \Rd)$.
\end{lemma}
\begin{proof}
Let $(t_0,x_0)\in (0,\infty)\times \Rd$. We have
\begin{eqnarray*}
	|u(t_0,x_0) - u(t,x)|
	& \leq & \int_\Rd |\tilde{p}_{t_0}(x_0,y) - \tilde{p}_t(x,y)| |f(y)|\, dy \\
	& \leq & e^{-\lambda t_0}\int_\Rd |p_{t_0}(x_0,y) - p_{t_0}(x,y)| |f(y)|\, dy \\
	&      &  +\, \int_\Rd |\tilde{p}_{t_0}(x,y) - \tilde{p}_t(x,y)| |f(y)|\, dy \to 0,
\end{eqnarray*}
as $(t,x) \to (t_0,x_0)$. This follows from the dominated convergence, since Theorem \ref{t-prop} for $|x-x_0|<t_0^{1/\alpha}$ yields
\begin{eqnarray*}
	\left| p_{t_0}(x_0,y)-p_{t_0}(x,y) \right|
	& \leq & c_1 \left(\frac{|x_0-x|}{t_{0}^{1/\alpha}}\right)^{\theta}e^{ct_0}
      \left(G_{t_0}^{(\alpha+\gamma)}(y-x_0)+
  G_{t_0}^{(\alpha+\gamma)}(y-x)\right) \\
	& \leq & c_2 \left(\frac{|x_0-x|}{t_{0}^{1/\alpha}}\right)^{\theta}e^{ct_0}
      G_{t_0}^{(\alpha+\gamma)}(y-x_0),
 \end{eqnarray*}
 and for $|t-t_0| \leq t_0/2$ and some $s\in (t\wedge t_0, t \vee t_0)$ we have
 \begin{eqnarray*}
   |\tilde{p}_{t_0}(x,y) - \tilde{p}_t(x,y)|
	 &   =  & \left| \left(e^{-\lambda s}\partial_s p_s(x,y) - \lambda e^{-\lambda s} p_s(x,y)\right)(t_0-t) \right| \\
	 & \leq & c_3 (s^{-1} + \lambda) e^{(c-\lambda)s} G_s^{(\alpha+\gamma)}(y-x) |t_0-t| \\
	 & \leq & c_4 (2/t_0 + \lambda) e^{(c-\lambda)(t_0/2)}
	   G_{t_0}^{(\alpha+\gamma)}(y-x) |t_0-t|.
 \end{eqnarray*}
If $(t,x) \to (0,x_0)$ for some $x_0\in\Rd$, then by Theorem \ref{t-exist} and the continnuity of $f$,
\begin{eqnarray*}
  | f(x_0) - u(t,x)|
	& \leq & |f(x_0) - f(x)| + |f(x) - {u}(t,x)| \to 0,
\end{eqnarray*}
This gives the continuity of $u$ on $[0,\infty)\times \Rd$.

Let $\delta>0$. Using the notation from Section~\ref{sec:iar}, by Fubini's theorem we get
\begin{eqnarray*}
  L^\delta {u} (t,x)
	  =   \int_{|u|>\delta} \left( {u}(t,x+u)-{u}(t,x)\right)\, \nu(x,du)
	 =   \int L^\delta_x \tilde p_t(x,y) f(y) dy.
\end{eqnarray*}
By \eqref{L3-est}  and the dominated convergence theorem we get
\begin{equation}\label{L_tu}
  L {u} (t,x) = \int L_x \tilde{p}_t(x,y) f(y)\, dy.
\end{equation}
In order to show that $\partial_t {u}(t,x) = \int \partial_t \tilde{p}_t(x,y)f(y)\, dy$, it suffices to estimate $|\partial_t \tilde{p}_t(x,y)|$ for every $t_0$, some $\delta>0$ and all $t\in (t_0-\delta,t_0+\delta)$ by an integrable function
depending only on $t_0$ and $x$, $y$. We obtain the estimate by using Theorem~\ref{t-prop}, which yields
\begin{eqnarray*}
  |\partial_t \tilde{p}_t(x,y)| \nonumber
	& \leq & \lambda e^{-\lambda t} | p_t(x,y) | + e^{-\lambda t} | \partial_t p_t(x,y)| \\ \nonumber
	& \leq & \lambda c_1 e^{-(\lambda-c) t} G_t^{(\alpha+\gamma)}(y-x) (1+t^{-1}) \\ \nonumber
	& \leq & \lambda c_1 G_t^{(\alpha+\gamma)}(y-x) (1+(t_0-\delta)^{-1}).
\end{eqnarray*}
By the dominated convergence theorem,
\begin{equation}\label{par_t_tu}
  \partial_t {u}(t,x) = \int \partial_t \tilde{p}_t(x,y)f(y)\, dy.
\end{equation}
We note here that \eqref{L_tu} and \eqref{par_t_tu} hold in fact for every bounded function $f$.
Now it easily follows from Theorem~\ref{t-exist} that
\begin{eqnarray*}
  (\partial_t - L_x) \tilde{p}_t(x,y)
	&   =  &  e^{-\lambda t}\partial_t p_t(x,y) - \lambda e^{-\lambda t} p_t(x,y) - e^{- \lambda t} L_x p_t(x,y) \\
	&   =  & -\lambda\tilde{p}_t(x,y),
\end{eqnarray*}
which, together with \eqref{par_t_tu} and \eqref{L_tu}, yield \eqref{pt_L_tu}.
If $\lambda\ge c$, then we have
\begin{eqnarray*}
  \left|\int_{\Rd} \tilde{p}_t(x,y) f(y)\, dy \right|
	& \leq &  c_1 e^{-(\lambda-c)t} \int_{\Rd} G_t^{\alpha+\gamma}(y-x) |f(y)|\, dy \\
	&  =   & c_1 e^{-(\lambda-c)t} \int_{\Rd} G_1^{\alpha+\gamma}(y) |f(t^{1/\alpha}y+x)|\, dy \\
	& \leq & c_1 e^{-(\lambda-c)t} \|f\|_{\infty} \int_{\Rd} G_1^{\alpha+\gamma}(y) \, dy \leq c_2 \|f\|_{\infty}.
\end{eqnarray*}
In fact, $|f(t^{1/\alpha}y+x)| \to 0$ as $t\to\infty$ or $|x|\to\infty$. By
the dominated convergence theorem, $\lim_{|(t,x)|\to\infty} {u}(t,x) = 0$.
\end{proof}

\begin{lemma}\label{mm}
If $u(t,x)\in C_0([0,\infty)\times \Rd)$, $\lambda\ge 0$ and
$(\partial_t-L_x+\lambda)u(t,x)=0$ on $(0,\infty)\times \Rd$, then
$$
\sup\limits_{(t,x)\in [0,\infty)\times \Rd}|u(t,x)|=\sup\limits_{x\in \Rd}|u(0,x)|.
$$
\end{lemma}
\begin{proof}
Let $m=\inf_{(t,x)\in [0,\infty)\times \Rd}u(t,x)$ and $M=\sup_{(t,x)\in [0,\infty)\times \Rd}u(t,x)$.
We have $-\infty<m\le 0\le M< \infty$.
If $M>0$ and $u(t_0,x_0)=M$ for some $t_0>0$ and $x_0\in \rd$, then
$\partial_t u(t_0,x_0) = 0$ and $L_x u(t_0,x_0) < 0$ by the maximum principle from Section~\ref{sec:iar}. Hence
$(\partial_t - L_x+\lambda) u(t_0,x_0)>0$.
This is a contradiction with the assumptions of the lemma, hence $M=0$ or the supremum of $u$ is attained at some boundary point $(0,x_0)$.
Similarly, if $m<0$ and $u(t_0,x_0)=m$ for some $t_0>0$ and $x_0\in \rd$, then
$\partial_t u(t_0,x_0) = 0$, $L_x u(t_0,x_0) > 0$, hence
$(\partial_t - L_x+\lambda) u(t_0,x_0) <0$.
Again, we conclude that $m=0$ or the infimum of $u$ is attained at some point $(0,x_0)$.
\end{proof}

\begin{corollary}\label{cu}
Let $\lambda\ge 0$. There is at most one solution $u\in C_0([0,\infty)\times \rd)$ to the Cauchy problem for $L-\lambda$.
\end{corollary}
\begin{proof}
By Lemma~\ref{mm}, the difference of two such solutions is zero
on $[0,\infty)\times \rd$.
\end{proof}

 \begin{lemma}\label{ChK}
$p$ is nonnegative and  satisfies the Chapman-Kolmogorov equation.
 \end{lemma}
\begin{proof}
 Let, as usual, $\tilde p_t(x,y)=e^{-\lambda t}p_t(x,y)$ and pick $\lambda\ge c$, the constant  in {\rm Theorem~\ref{t-prop}}.
Let $f\in C_0(\Rd)$. By Lemma~\ref{nn},
$u(t,x):=\int \tilde p_t(x,y)f(y)dy$ extends to a function of the class $C_0([0,\infty)\times \Rd)$.
Recall that $p$ is continuous (see Lemma \ref{p_continuity}), hence $\tilde p$ is continuous.
Considering that  all the nonnegative functions $f\in C_0(\Rd)$ are allowed here, by the proof of Lemma~\ref{mm} we get that $\tilde p\ge 0$ and so $p\ge 0$.

We next consider $s>0$ and $u(s,x)$ defined above. For $t\ge 0$, $x\in \rd$, let $ w(t,x)$ be the solution to the Cauchy problem for $L-\lambda$ with the initial condition $w(0,x)=u(s,x)$, $x\in \Rd$. By Lemma~\ref{nn} and Corollary~\ref{cu},
$$\int_\Rd \tilde p_{s+t}(x,y)f(y)dy=
u(s+t,x)= w(t,x)=\int_\Rd \tilde p_t(x,y)\int_\Rd \tilde p_s(y,z)f(z)dzdy.$$
Since $f\in C_0(\Rd)$ is arbitrary, using Fubini's theorem we see that $\tilde p$ satisfies the Chapman-Kolmogorov equation and so does $p$.
\end{proof}
For $f\in C_0(\Rd)$, $t>0$ and $x\in \rd$, we let
$$\tilde P_t f(x)=\int_\Rd \tilde p_t(x,y)f(y)dy.$$
We conclude that $\{\tilde P_t\}$ and $\{P_t\}$ are
strongly continuous semigroups on $C_0(\Rd)$.
\begin{lemma}\label{sps}
If $\lambda\ge c$, the constant  in {\rm Theorem~\ref{t-prop}}, then $\{\tilde P_t\}$ is sub-Markovian.
\end{lemma}
\begin{proof}
By Lemma~\ref{ChK}, $\tilde P_t f\ge 0$ if $f\in C_0(\rd)$ and $f\ge 0$. By Lemma~\ref{nn} and Lemma~\ref{mm}, $\|P_tf\|_\infty\le \| f\|_\infty$, as needed.
\end{proof}
In particular, if $\lambda\ge c$, then for all $t>0$ and $x\in \rd$ we have
$
\int_\rd \tilde p_t(x,y)dy\le 1.
$
To prove the Markovianity in Theorem~\ref{t-uni} we need to verify that
$\int_\rd p_t(x,y)dy\equiv 1$.
The result requires preparation.
Let $\mathcal L$ be the generator of $\{P_t\}$. Then $\mathcal L-\lambda$ is the generator of $\{\tilde P_t\}$, with the same domain, say $D(\mathcal L)$, a dense subset of $C_0(\Rd)$. We will make a connection between $L$ and $\mathcal L$.
Let $\phi \in C_0(\Rd)$, $0<T<\infty$ and $f=\int_0^T P_s \phi ds$.
 By the general semigroup theory, $f\in
D(\mathcal{L})$
and
$\partial_t P_t f= {\mathcal L}P_t f\in C_0(\Rd)$ for every $t>0$. By Lemma~\ref{nn}, $\partial_t P_t f(x)= L P_t f(x)$
for all $t>0$ and $x\in \Rd$, hence ${\mathcal L}P_t f=L P_t f$ for all such $t$ and $f$. Therefore $L={\mathcal L}$ on a dense subset of $C_0(\Rd)$.

The following more explicit result is rather delicate.
\begin{theorem}
$  \mathcal{L} f=L f$ for $f\in C_0^2(\rd)$.
\end{theorem}
 \begin{proof}
We first prove that for
H\"older continuous function $g\in C_0(\rd)$ we have
\begin{equation}\label{change1}
L \int_0^t P_s g(x)ds= \int_0^t L P_s g(x)ds, \quad x\in \Rd.
\end{equation}
Indeed, for $\delta>0$ the operator $L^\delta$ is bounded and linear  on $C_0(\rd)$, hence
\begin{align*}
L \int_0^t   P_s g(x)ds &= \lim_{\delta\to 0} L^\delta \int_0^t P_s g(x)ds=  \lim_{\delta\to 0}  \int_0^t  L^\delta P_s g(x)ds\\
&= \lim_{\delta\to 0}  \int_0^t  \int_\rd L^\delta_x p_s (x,y) g(y) dy ds= I+ I\!I,
\end{align*}
where
$$I=\lim_{\delta\to 0}  \int_0^t  \int_\rd L^\delta_x p_s (x,y) [g(y)- g(x)]  dy ds, \quad I\!I= g(x)  \lim_{\delta\to 0}  \int_0^t  \int_\rd L^\delta_x p_s (x,y) dy ds$$
are finite, as we will shortly see.
For $I\!I$, by Corollary~\ref{rem-6-1} and Lemma~\ref{lem-ltil} we have
$$
\lim_{\delta\to 0}   \int_\rd L^\delta_x p_s (x,y) dy=   \int_\rd \lim_{\delta\to 0}L^\delta_x p_s (x,y) dy=
\int_\rd L_x p_s (x,y) dy.
$$
Therefore by \eqref{Ldel1} and the dominated convergence theorem,
\begin{align*}
I\!I
&=g(x)
\int_0^t  \lim_{\delta\to 0}  \int_\rd L^\delta_x p_s (x,y) dy ds
=g(x)\int_0^t\Big(\int_\rd L_x p_s (x,y) dy\Big) ds.
\end{align*}
It is important to notice that the last (outer) integral $\int_0^t (\ldots)ds$ converges absolutely.
We now turn to $I$.
Let $\epsilon>0$ be such that $\alpha+\gamma-\epsilon>d$. Assume that $g$ is H\"older continuous of order $\epsilon$. Therefore, for all $x,y\in \rd$ and $s\in (0,t)$, by \eqref{L3-est}   we get   
\begin{align*}
  \big| L^\delta_x &p_s (x,y) [g(y)- g(x)]  \big| \le
  c |L^\delta_x p_s (x,y) |  (1\wedge |x-y|^\epsilon)  \\
  &  \leq c s^{-1} \big(G_s^{\alpha+\gamma}(y-x) + s^{\theta/\alpha}  K_s(x,y)\big)
	  (1\wedge |x-y|^\epsilon)  
  \\
  & \leq cs^{-1+ \epsilon/\alpha}   G_s^{(\alpha+\gamma-\epsilon)}(y-x)+ s^{-1+\theta/\alpha}  K_s(x,y).
\end{align*}
The above expression is integrable in $dyds$.    Of course, $ \lim_{\delta \to 0 }L^\delta_x p_s (x,y)=L_x p_s (x,y)$. By the dominated convergence theorem,
$$
I=\int_{(0,t)\times\rd} L_x p_s (x,y) [g(y)- g(x)]  dyds,
$$
which is finite. Adding $I$ and $I\!I$ we obtain
\begin{equation*}\label{change12}
L \int_0^t P_s g(x)ds= \int_0^t\Big(\int_\rd  L_x p_s(x,y) g(y)dy\Big)ds.
\end{equation*}
By \eqref{L3-est}  and the boundedness of $L^\delta$ we have that
\begin{align*}
\int_\rd  L_x p_s(x,y) g(y)dy&=\int_\rd  \lim_{\delta\to 0} L^\delta _x p_s(x,y) g(y)dy
=\lim_{\delta\to 0}\int_\rd   L^\delta _x p_s(x,y) g(y)dy\\&=\lim_{\delta\to 0}L^\delta\int_\rd    p_s(x,y) g(y)dy
=\lim_{\delta\to 0}L^\delta P_s g(x).
\end{align*}
Therefore, $\int_\rd  L_x p_s(x,y) g(y)dy=L P_s g(x)$, which gives \eqref{change1}.

We next claim that for $f\in C_0^2(\rd)$, $0<t<\infty$ and $x\in \Rd$,
\begin{equation}\label{e:fTc}
P_t f(x)- f(x) = \int_0^t P_s L f(x)ds.
\end{equation}
To prove \eqref{e:fTc} we let
$\lambda> c>0$, with $c$ from {\rm Theorem~\ref{t-prop}}, and we define
$$
u(t,x)=e^{-\lambda t } \Big[  P_{t} f(x)-  f(x)-\int_0^t P_{s} L f(x)ds\Big].
$$
We also let $u(0,x)=0$. By Lemma~\ref{nn}, $u\in C_0([0,\infty)\times \rd)$.
By \eqref{change1} with $g=Lf$,
\begin{align*}
(\partial_t - L )u(t,x)&=  -\lambda u(t,x)+ e^{-\lambda t} \Big[L f(x)-  P_{t}  L f(x) + \int_0^t L P_{s} L f(x)ds\Big].
\end{align*}
From the discussion of \eqref{change1} the last integral is absolutely convergent,
implying that  $\partial_s P_{s} L f(x)=L P_{s} L f(x)$ is also absolutely integrable, cf. Lemma~\ref{nn}. Therefore,
\begin{equation*}\label{lamu}
(\partial_t - L )u(t,x)= - \lambda u(t,x)+e^{-\lambda t }\left[-\!\!\int_0^t \!\!\partial_s P_{s} L f(x)ds + \int_0^t \!\!L P_{s} L f(x)ds\right]= - \lambda u(t,x).
\end{equation*}
We now prove that $u\equiv 0$. Recall that $u\in C_0([0,\infty)\times \rd)$. If $u$ attains a strictly positive maximum at some point $(t_0,x_0)\in (0,\infty)\times \rd$, then   $(\partial_t - L )u(t_0,x_0)=-\lambda u(t_0,x_0)<0$, but  the maximum principle for $L$ contradicts this: $(\partial_t - L)u(t_0,x_0)= - L u (t_0, x_0)> 0$. Therefore we must have $u\le 0$. Analogously we prove that $u\le 0$ and so $u=0$ everywhere.

Finally, we divide both sides of \eqref{e:fTc} by $t$ and let $t\to 0$. We obtain $\mathcal L f(x)= L f(x)$. The proof is complete: the operator $L$ and the generator $\mathcal L$ coincide on $C_0^2(\rd)$.
\end{proof}

\subsection{Proof of Theorem~\ref{t-uni}}
We only need to prove
that for all $t>0$ and $x\in \Rd$ we have
$\int_\Rd p_t(x,y)dy=1$.
  We know that the operators $P_t f(x) = \int_{\Rd} p_t(x,y) f(y)\, dy$, $t>0$,
  form a strongly
	continuous semigroup on $C_0(\Rd)$ with the generator $\mathcal L$.
	We fix $t>0$. If $f$ is in the domain of $\mathcal L$, then
	from  the general theory of semigroups (see \cite[Ch. 2, Lemma 1.3]{MR1721989} or \cite[Lemma 4.1.14]{MR1873235}),
	$$
	  P_t f(x) - f(x) = \int_0^t P_s \mathcal L f(x)\, ds.
	$$
	In particular, let $f\in C^2_0(\Rd)$ be such that $|f(x)|\leq 1$ for all $x\in\Rd$ and
	$f(x)=1$ for $|x|<1$. Let $f_n(x)=f(x/n)$, $n\ge 1$.
We have $\lim_{n\to\infty} f_n(x)=1$ and $\lim_{n\to\infty} P_t f_n(x) = \int_{\Rd} p_t(x,y)\, dy$, which easily follows from bounded convergence.  Furthermore,
	\begin{equation}\label{dyn}
	  P_t f_n(x) - f_n(x) = \int_0^t P_s L f_n(x)\, ds.
	\end{equation}
	If $x\in \rd$ is fixed and $n>2|x|$, then
	\begin{eqnarray*}
	  |Lf_n(x)|
		&  =   & \frac{1}{2}\left| \int_{\Rd} \left(f_n(x+u)+f_n(x-u)-2f_n(x) \right)\, \nu(x,du)\right| \\
		&  =   & \frac{1}{2}\left| \int_{|u|>n/2} \left(f((x+u)/n)+f((x-u)/n) -2 \right)\, \nu(x,du) \right| \\
		& \leq & \nu(x,B(0,n/2)^c) \leq M_0 \nu_0(B(0,n/2)^c) \leq c_1 n^{-\alpha}.
	\end{eqnarray*}
	This yields
	$$
	  \left| \int_0^t P_s L f_n(x)\, ds \right| \leq \int_0^t P_s |L f_n(x)|\, ds
		\leq c_2 t n^{-\alpha} \to 0,
	$$
	as $n\to\infty$.
	By \eqref{dyn} and the above  discussion we get $\int_{\Rd} p_t(x,y)\, dy = 1$.
The proof  of Theorem~\ref{t-uni} is complete. \qed{}

We end the paper by pointing out in which sense $p_t(x,y)$ is unique. Plainly, if 
$\mathfrak{p}_t(x,y)$
has the properties listed in Theorem~\ref{t-exist} and \ref{t-prop}, then $\mathfrak{p}_t(x,y)=p_t(x,y)$ for all $t>0$, $x,y\in \Rd$.
Indeed, let $s>0$ and $z\in \rd$. By the proof of Lemma~\ref{nn},
 $u(t,x):=e^{-\lambda t}\int_\Rd \mathfrak{p}_t(x,y)p_s(y,z)dy$ and $\mathfrak u(t,x):=e^{-\lambda t}\mathfrak{p}_{t+s}(x,z)$ give solutions to the same Cauchy problem for $L-\lambda$, and they are in $C_0([0,\infty)\times \rd)$ for large $\lambda>0$.
 By Corollary~\ref{cu},
$$
\int_\Rd \mathfrak{p}_t(x,y)p_s(y,z)dy=\mathfrak{p}_{t+s}(x,z), \quad s,t>0,\ x,y\in \Rd.
$$
We claim that for all $f\in C_0(\Rd)$, uniformly in $x\in \rd $ we have
\begin{equation}\label{f-solx}
   \lim_{t\to 0} \int_\Rd f(x) p_t(x,y)\, dx =f(y).
\end{equation}
For clarity, this is different than \eqref{f-sol}. To prove \eqref{f-solx} we note that
$$
\int_\rd p^0_t(x,y)dx=\int_\rd p^y(y-x)dx=1, \quad t>0,\ y\in \rd,
$$
we recall \eqref{r}, \eqref{e:small}, \eqref{H30}, Lemma~\ref{th-ptx2} with $\beta=0$, the scaling of $G_t^{(\alpha+\gamma)}$ and the dominated convergence.
By \eqref{f-solx} we get $p_s(x,z)=\mathfrak{p}_s(x,z)$ for all $s>0$, $x,z\in \rd$. 


\end{document}